\documentclass[11pt]{amsart}
\usepackage[colorlinks=true, pdfstartview=FitV, linkcolor=blue, 
citecolor=blue]{hyperref}

\usepackage{amssymb,amscd}
\usepackage{graphicx}
\usepackage{a4wide}
\usepackage{mathtools}

\theoremstyle{plain}
\newtheorem{theorem}{Theorem}[section]
\newtheorem{prop}[theorem]{Proposition}
\newtheorem{lemma}[theorem]{Lemma}
\newtheorem{coro}[theorem]{Corollary}
\newtheorem{fact}[theorem]{Fact}

\theoremstyle{definition}
\newtheorem{remark}[theorem]{Remark}

\newcommand{\dd}{\,\mathrm{d}}
\newcommand{\ts}{\hspace{0.5pt}}
\newcommand{\nts}{\hspace{-0.5pt}}
\newcommand{\pa}{\phantom{a}}

\newcommand{\NN}{\mathbb{N}}
\newcommand{\RR}{\mathbb{R}}
\newcommand{\ZZ}{\mathbb{Z}}
\newcommand{\XX}{\mathbb{X}}
\newcommand{\TT}{\mathbb{T}}

\newcommand{\cB}{\mathcal{B}}
\newcommand{\cF}{\mathcal{F}}
\newcommand{\cG}{\mathcal{G}}
\newcommand{\cM}{\mathcal{M}}
\newcommand{\cP}{\mathcal{P}}

\newcommand{\ee}{\mathrm{e}}

\DeclareMathOperator{\card}{card}

\DeclareMathOperator{\locdim}{dim}

\newcommand{\exend}{\hfill $\Diamond$}

\newcommand{\myfrac}[2]{\frac{\raisebox{-2pt}{$#1$}}
      {\raisebox{0.5pt}{$#2$}}}

\newcommand{\birk}{b} 
\newcommand{\dims}{f}

\begin{document}

\title{Scaling properties of the Thue{\ts}--Morse measure}

\author{Michael Baake, Philipp Gohlke}
\address{Fakult\"{a}t f\"{u}r Mathematik, Universit\"{a}t Bielefeld,\newline
\hspace*{\parindent}Postfach 100131, 33501 Bielefeld, Germany}
\email{$\{$mbaake,pgohlke$\}$@math.uni-bielefeld.de}

\author{Marc Kesseb\"{o}hmer}
\address{Fachbereich Mathematik, Universit\"{a}t Bremen, \newline 
\hspace*{\parindent}Postfach 330440, 28359 Bremen, Germany}
\email{mhk@math.uni-bremen.de}

\author{Tanja Schindler}
\address{Research School of Finance, Actuarial Studies
      and Statistics,\newline
\hspace*{\parindent}Australian National University, \newline 
\hspace*{\parindent}26C Kingsley St, Acton ACT 2601, Australia }
\email{tanja.schindler@anu.edu.au}

\begin{abstract}
  The classic Thue{\ts}--Morse measure is a paradigmatic example of a
  purely singular continuous probability measure on the unit
  interval. Since it has a representation as an infinite Riesz
  product, many aspects of this measure have been studied in the past,
  including various scaling properties and a partly heuristic
  multifractal analysis. Some of the difficulties emerge from the
  appearance of an unbounded potential in the thermodynamic formalism.
  It is the purpose of this article to review and prove some of the
  observations that were previously established via numerical or
  scaling arguments.
\end{abstract}



\maketitle
\thispagestyle{empty}

\section{Introduction}

As is well known, see \cite[Sec.~10.1]{TAO} and references therein,
the Thue{\ts}--Morse diffraction measure for the balanced-weight case
is given by the infinite Riesz product
\[
  \mu^{\pa}_{\mathrm{TM}} \, =  \prod_{\ell = 0}^{\infty}
    \bigl(1 - \cos(2 \pi 2^{\ell} k)\bigr) ,
\]
where convergence is understood in the vague topology. As such,
$\mu^{\pa}_{\mathrm{TM}}$ is a translation-bounded, positive measure
on $\RR$ that is purely singular continuous and $1$-periodic.  Clearly,
\[
   \mu^{\pa}_{\mathrm{TM}} \, = \, \nu \ast \delta^{\pa}_{\ZZ} \ts ,
\]
with $\nu = \mu^{\pa}_{\mathrm{TM}} \rvert_{[0,1)}$ being a
probability measure on $\TT = \RR/\ZZ$, the latter represented by
$[0,1)$ with addition modulo $1$. More precisely, $\nu$ is the weak
limit of probability measures $\nu^{\pa}_{\nts N}$ with Radon--Nikodym
densities
\begin{equation}\label{eq:TM-def}
  \frac{\dd \nu^{\pa}_{\nts N}}{\dd \lambda}(k) \, =
  \prod_{\ell = 0}^{N-1} \bigl(1 - \cos(2\pi 2^{\ell} k)\bigr)  ,
\end{equation}
relative to Lebesgue measure $\lambda$.  In this setting,
$\nu$ is a natural choice for
the maximal spectral measure in the orthocomplement of the pure point
sector of the Thue{\ts}--Morse dynamical system \cite{Q}.  Since $\nu$
is a continuous measure, it is often advantageous to simultaneously
consider $\nu$ as a measure on $\TT$ and on $[0,1]$, as we shall see
later several times.  It is the aim of this paper to obtain
information about the local structure of the Thue{\ts}--Morse measure
from multifractal analysis and its relation to thermodynamic
formalism. As there are many singular continuous measures
with similar properties that are of interest in number theory, see
\cite{BC} for a recent example, we hope that this approach will prove
useful there as well, as it did in \cite{KS}.

Let $\varrho^{\pa}_{\mathrm{E}}$ denote the Euclidean metric on $\TT$,
$B_{\mathrm{E}}(x,r)$ the closed ball around $x \in \TT$ with
Euclidean radius $r$, and let $B_{2}(x,r)$ denote the closed ball
around $x \in \TT$ with radius $r$ with respect to the shift space
metric $\varrho^{\pa}_2$, where $\varrho^{\pa}_2(x,y) := 2^{-k}$ when
$k \in \NN_0$ is the largest integer such that $x_i = y_i$ for the dyadic
digits of $x$ and $y$, for all
$i \leqslant k$.  We note here that this metric is not well defined
for the dyadic points $x \in \TT$; see our discussion in
Section~\ref{subsec:norms}.  However, as the dyadic points form only a
countable subset of $\TT$, we do not consider those points any
further. A discussion about the differences between the two metrics
will also be given in Section~\ref{subsec:norms}.

One way to quantify how concentrated the measure $\nu$ is at a given
point $x \in \TT$ is to determine its \emph{local dimension}, given by
\[
  \dim_{\nu, \tau}(x) \, =
  \lim_{r \to 0} \frac{\log \nu(B_{\tau}(x,r))}{\log(r)}
  \qquad \text{with } \, \tau \in \{ \mathrm{E}, 2 \} \ts ,
\]  
provided that the limit exists.  Such a limit then also applies to the
local dimension of $\mu^{\pa}_{\mathrm{TM}}$ at any $x + n$ with
$n \in \ZZ$.  Due to their highly irregular structure, we cannot hope
to pin down the level sets of $\dim_{\nu, \tau}$ explicitly. However, the
corresponding \emph{Hausdorff dimension} with respect to the metric
$\varrho^{\pa}_{\tau}$,
\[
  \dims^{\pa}_{\tau} (\alpha) \, = \,
  \dim^{\pa}_{ \mathrm{H},\tau} \{ x \in \TT : \dim_{\nu , \tau}(x)
  = \alpha \} \ts ,
\]
yields a properly behaved function of $\alpha$.  The analysis of
$\dims^{\pa}_{\tau} (\alpha)$ is one of the open questions considered
in \cite{Q1}.  The problem to determine the local dimension at a given
point $x \in \TT$ turns out to be intimately related to pointwise
scaling properties of the approximants in Eq.~\eqref{eq:TM-def}. More
precisely, we consider
\[
  \beta(x) \, :=  \lim_{n \to \infty}
  \frac{1}{n \log(2)} \, \log \prod_{\ell = 0}^{n-1}
  \bigl( 1 - \cos(2^{\ell+1} \pi x) \bigr) ,
\]
for all $x \in \TT$ for which the limit exists. Since
$\prod_{\ell = 0}^{n-1} \bigl( 1 - \cos(2^{\ell+1} \pi x) \bigr)$ is
related to the diffraction measure resulting from the first $2^n$
letters of the Thue{\ts}--Morse chain, we obtain some physical
interpretation of the quantity $\beta(x)$. By standard results,
compare \cite{BGN} and references therein, it is known that the
scaling exponent $\beta(x)$ exists and equals $-1$ for Lebesgue-a.e.\
$x \in \TT$. For some particular examples of non-typical points, see
\cite{CSM,BGN}.

There is a natural way to interpret $\beta$ in terms of the Birkhoff
average of some function
$\psi \colon \TT \xrightarrow{\quad} [-\infty, \log(2)]$,
\begin{equation}\label{eq:psi-def-1}
  \psi(x) \,=\, \log \bigl( 1 - \cos(2 \pi x) \bigr),
  \quad \beta(x) \,= \lim_{n \rightarrow \infty}
  \frac{ \psi_n(x)}{n \log(2)} \ts ,
\end{equation}
where $\psi^{\pa}_n(x) = \sum_{\ell = 0}^{n-1} \psi(2^\ell
x)$. Considered as a Borel probability measure on the dynamical system
$(\TT,T)$ with $T(x) = 2x$ (mod $1$), $\nu$ is an \emph{equilibrium
  measure} for the thermodynamic potential $\psi$ --- in fact, it is a
$g$-measure in the sense of Keane \cite{keane71}; see the explanation
in Section~\ref{Sec:Eq-measure}. In analogy to known results for
H\"{o}lder continuous potentials \cite[Cor.~1]{PesinWeiss}, we expect
to find some simple relation between $\dims^{\pa}_{\tau} (\alpha)$ and
the \emph{Birkhoff spectrum}
\begin{equation}\label{eq:Birk-spec-and-level-sets}
\begin{split}  
  \birk_{\tau} (\alpha) \, & = \, \dim_{ \mathrm{H},\tau} \cB(\alpha)
  \qquad \text{for } \tau \in \{ \mathrm{E}, 2 \} \ts ,  \quad
   \text{with} \\[1mm]
  \cB(\alpha) \, &  = \, \Bigl\{ x \in \TT : \lim_{n \to \infty}
  \frac{\psi_n(x)}{n} = \alpha \Bigr\} \, = \,
  \Bigl\{x \in \TT : \beta(x) =s
  \frac{\alpha}{\log(2)} \Bigr\} ,
\end{split}  
\end{equation}
where $\tau$ indicates the metric type.  It is one of the strengths of
the thermodynamic formalism to connect such locally defined functions
to the Legendre transform of a globally defined quantity. An adequate
choice for the latter in our situation is the \emph{topological
  pressure} of the function $t\ts \psi$, $t\in \RR$, defined by
\begin{equation} \label{eq:def-pressure}
  p(t) \,  := \, \cP (t\ts \psi)
  \, :=  \lim_{n \to \infty} \myfrac{1}{n} \log
  \sum_{J\in I_n} \sup_{x \in J} \, \exp \bigl(t\ts \psi^{\pa}_n(x)\bigr) ,
\end{equation}
where, for each $n\in \NN$, $I_n$ forms a partition of $[0,1]$
into intervals of length $2^{-n}$.

With this at hand, we can state our main result as follows, where
$p^{*}$ denotes the Legendre transform of $p$.

\begin{theorem}\label{thm:MAIN-MF-1}
  The Birkhoff spectrum of the function\/ $\psi$ from
  \eqref{eq:psi-def-1} is given by
\begin{equation}\label{eq:Birkhoff-spectrum}
  \birk_{\tau} (\alpha ) \, = \, \birk \ts (\alpha) \, = \,
  \max \Bigl\{ \frac{-p^{*}\ts (\alpha)}{\log(2)}, 0 \Bigr\} ,
\end{equation}
for\/ $\tau \in \{\mathrm{E},2\}$.  The function\/ $\birk$ is concave
on\/ $(-\infty,\log(3/2)]$, constantly equal to\/ $1$ for all\/
$\alpha \leqslant - \log(2)$, strictly less than\/ $1$ for
$\alpha > - \log(2)$, and equal to\/ $0$ for\/
$\alpha\geqslant \log(3/2)$; see
Figure~\textnormal{\ref{fig:The-graph-of-f(alpha)}} for the graph of
the spectrum. Moreover, the level sets\/ $\cB (\alpha)$ are empty
for\/ $\alpha > \log(3/2)$.

Finally, the dimension spectrum of the
measure\/ $\nu$ is related to the Birkhoff spectrum by
\begin{equation}\label{eq:dimension-spectrum}
  \dims^{\pa}_{\tau}(\alpha) \, = \,  \dims (\alpha) 
  \, = \, \birk \bigl(\log(2) (1 - \alpha)\bigr) .
\end{equation}
\end{theorem}

\begin{figure}
  \includegraphics[width= 0.7\textwidth]{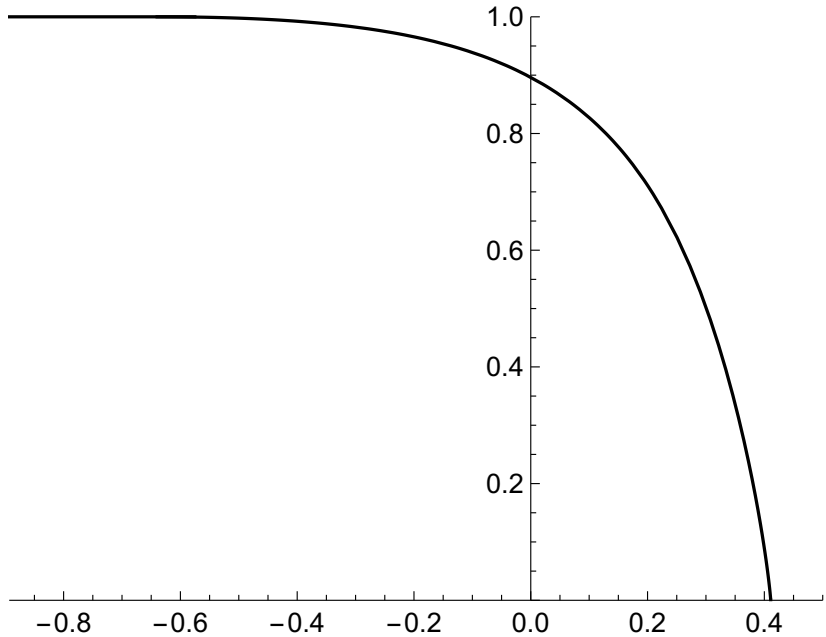}
  \caption{The graph of the Birkhoff spectrum $b$ from
    Eq.~\eqref{eq:Birk-spec-and-level-sets}.
  \label{fig:The-graph-of-f(alpha)}}
\end{figure}

The main technical difficulty in proving this result can be traced
back to the fact that the potential $\psi$ exhibits a singularity at
$0 \in \TT$. Since the authors are not aware of any general framework
covering this case, we choose to present not only the general line of
arguments but also the individual steps in detail.  One major part in
proving this result is the calculation of the pressure function given
in Eq.~\eqref{eq:def-pressure}.  This can be seen as a tractable
example for a pressure function with an unbounded potential for which
we establish an interesting and important approximation result in
Section~\ref{Sec:Restricted-pressure}.  We will present the crucial
properties of $p$ needed for the proof of our main theorem in Section
\ref{Sec:p-properties}; see \cite{KimLiaoRamsWang17} for related
results on non-integrable locally constant potentials with
singularities.\footnote{We would like to mention that J{\"o}rg
  Schmeling et al.\ are independently working on general results for
  locally constant potentials with singularities.}  Finite-size
scaling arguments for the validity of Theorem~\ref{thm:MAIN-MF-1}, and
a numerical approach that is equivalent to
Figure~\ref{fig:The-graph-of-f(alpha)}, have been provided by
\cite{GL90}.  In \cite{fan-multifrac}, a general approach for the
multifractal analysis of measures generated by infinite products is
given, which does not allow for singularities in the potential
function, however.  Finally, we will use classical results to show
that $\nu$ is an \emph{equilibrium} measure.  In this context, we will
also give a precise numerical value for the metric entropy of $\nu$.
\smallskip

The paper is organised as follows.  In
Section~\ref{Sec:Preliminaries}, we introduce an alternative view to
relate $(\TT,T)$ to a symbolic shift space and establish some
properties of the potential $\psi$.  This enables us to prove the
maximal scaling exponent $\beta$ in
Section~\ref{Sec:Max-exponent}. Some properties that are reminiscent
of Gibbs measures are shown, in Section~\ref{Sec:Gibbs-type}, to hold
for $\nu$.  In Section~\ref{Sec:Restricted-pressure}, we show that the
full pressure function emerges as a limit of pressure functions
that are restricted to certain subshifts of finite type.  This allows
us to employ results on H\"{o}lder continuous potentials and thus to
give a proof of our main theorem in Section~\ref{Sec:Proof-of-Thm}.
Further properties of the pressure function are established in
Section~\ref{Sec:p-properties}, thus providing all necessary
properties to prove the remaining part of Theorem~\ref{thm:MAIN-MF-1}
in Section~\ref{Sec:Proof-of-Thm-II}, whereas
Section~\ref{Sec:Eq-measure} is devoted to embed the previous results
into the realm of equilibrium measures and the variational principle.

\section{Preliminaries and Notation}\label{Sec:Preliminaries}

\subsection{Norms and metrics}\label{subsec:norms}
Consider the ergodic dynamical system $(\TT, T, \lambda)$, with
mapping $T \colon x \mapsto 2x$ (mod $1$) and $\lambda$ the Lebesgue
measure on $\TT$.  This enables us to regard the function
$\psi^{\pa}_{n} = \sum_{\ell = 0}^{n-1} \psi \circ T^{\ell}$ as a
Birkhoff sum with ergodic transformation $T$.  It is sometimes more
convenient to consider the closed unit interval instead of $\TT$, with
or without identifying the endpoints --- which will be clear from the
context. When doing so, we extend $T$ from $\TT$ to $[0,1]$ in the
obvious way, with $T(1)=1$ and $T \left( \frac{1}{2} \right) =
1$. This is then consistent with the identification of $1$ and $0$.

Information about the orbit of a point $x\in \TT$ under the action of
$T$ is most easily obtained if we represent $x$ as a binary
sequence. Let $\XX = \{ 0,1 \}^{\NN}$ be the one-sided binary shift
space, and use $x = (x^{\pa}_{1}, x^{\pa}_{2}, \ldots )$ to denote a
sequence $x \in \XX$. By slight abuse of notation, we also use $x$ to
denote the corresponding number,
\[
     x \, = \sum_{i=1}^{\infty} x^{\pa}_{i} \ts\ts 2^{-i} ,
\]
where the actual meaning will always be clear from the context.  This
way, $x\in\XX$ is always a number in $[0,1]$, which we will often
simply write as the binary string
$x^{\pa}_{1} \ts x^{\pa}_{2} \cdots $, in line with
$2 \cdot x = x^{\pa}_{2} \ts x^{\pa}_{3} \cdots$ etc. Thus, the shift
action $\sigma$ on $\XX$ can be represented via the doubling map $T$
on $[0,1]$.

With the bit flip $\widehat{0}=1$ and $\widehat{1}=0$, which extends
to $x$ bitwise, one obtains $\widehat{x} = 1-x$ for any $x\in [0,1]$,
with the latter written in binary expansion.  For many of our
purposes, it does not matter that the dyadic numbers have two
representations as a sequence.  We also introduce notation for the set
of finite binary strings,
$\Sigma^{\star} = \bigcup_{n \in \NN} \{0,1\}^n$, and for the periodic
continuation of some $\omega \in \Sigma^{\star}$ as the infinite
sequence $\overline{\omega} = \omega \omega \omega \cdots$.  With
$q^{\pa}_{1} , \ldots , q^{\pa}_{n} \in \{ 0,1 \}$, we denote the
corresponding \emph{cylinder} for the string
$q^{\pa}_1 \! \cdots q^{\pa}_n$ as
\[
    \langle \ts q^{\pa}_{1} \! \cdots  q^{\pa}_{n} \rangle \, := \,
    \{ q^{\pa}_{1} \! \cdots q^{\pa}_{n} x : x\in \XX \ts  \} \ts ,
\]
where we employ the standard convention for the concatenation of
sequences.  In particular, by our identification, we have
$\langle 0 \rangle = \bigl[0, \frac{1}{2} \bigr]$ and
$\langle 1 \rangle = \bigl[ \frac{1}{2}, 1 \bigr]$.

Note that both the Hausdorff dimension and the local dimension of a
measure implicitly depend on the metric chosen.  Following our
discussion above, in addition to $( \TT,\varrho^{\pa}_{\mathrm{E}})$
with the Euclidean distance $\varrho^{\pa}_{\mathrm{E}}$, we also want
to consider $(\XX,\varrho^{\pa}_2)$, with $\varrho^{\pa}_2$ as defined
earlier.  We note that, even with the above identification of points
between $\XX$ and $[0,1]$, this metric is \emph{not} equivalent to
$\varrho^{\pa}_{\mathrm{E}}$. In particular, $0\overline{1}$ and
$1\overline{0}$ have Euclidean distance $0$ but distance $1$ in the
shift space. Nevertheless, $\varrho^{\pa}_2$ and
$\varrho^{\pa}_{\mathrm{E}}$ show some consistency in the sense that
both metrics assign the same length $2^{-n}$ to cylinders (intervals)
of the form $\langle q_1 \! \cdots q_n \rangle$.  This will be
sufficient to conclude that Theorem~\ref{thm:MAIN-MF-1} holds
irrespective of whether $\nu$ is considered as a measure on the metric
space $(\XX, \varrho^{\pa}_2)$ or on the space
$\bigl( [0,1],\varrho^{\pa}_{\mathrm{E}} \bigr)$.

\subsection{Basic properties of the measure
  \texorpdfstring{$\nu$}{nu}
  and the potential \texorpdfstring{$\psi$}{psi}}
Let $P^{\pa}_{0} :=1$ and set $P^{\pa}_{1} (x) := 1 - \cos (2 \pi
x)$. Now, for $n\geqslant 0$, define the trigonometric polynomial
$P^{\pa}_{n}$ on $\TT$ recursively by
\begin{equation}\label{eq:P-def}
    P^{\pa}_{n+1} (x) \,  :=  \, P^{\pa}_{n} (2 x) \ts P^{\pa}_{1} (x).
\end{equation}
Clearly, one has
$P^{\pa}_{n} (x) = \prod_{\ell=0}^{n-1} P^{\pa}_{1} (2^\ell x)$.
A simple calculation gives the following result.

\begin{fact}
  For each\/ $n\in\NN_{0}$, the trigonometric polynomial\/ $P_{n}$ is
  non-negative and defines a probability density, both on\/ $\TT$ and
  on\/ $[0,1]$.  \qed
\end{fact}

For our study of growth estimates and asymptotic properties, we
write
\begin{equation}\label{eq:psi-def}
   \psi^{\pa}_{n} (x) \, = \, \log \bigl( P_{n} (x) \bigr) .
\end{equation}
We abbreviate $\psi = \psi^{\pa}_{1}$, which has a unique maximum at
$x=\frac{1}{2}$ and singularities at $x=0$ and $x=1$ with value
$-\infty$. The central role of $\psi$ as a thermodynamic potential
will emerge when we discuss the variational principle in
Section~\ref{Sec:Eq-measure}. The first two derivatives of $\psi$ are
\begin{equation}\label{eq:psi-deriv}
  \psi^{\ts \prime} (x) \, = \,
   \frac{2 \ts \pi \, \sin(2 \pi x)}{1-\cos(2 \pi x)}
   \qquad \text{and} \qquad
   \psi^{\ts \prime \prime} (x) \, = \,
   \frac{- 4\ts \pi^{2}}{1-\cos(2 \pi x)} \ts , 
\end{equation}
where $\psi^{\ts \prime \prime} (x) < 0$.  This, together with
Eq.~\eqref{eq:P-def}, immediately implies the following result.

\begin{fact}\label{fact:psi-props}
  On\/ $[0,1]$, the function\/ $\psi$ is strictly concave and
  satisfies the symmetry relation\/ $\psi (1-x) = \psi (x)$.
  Moreover, for all\/ $n\in\NN$, the functions\/ $\psi^{\pa}_{n}$ from
  Eq.~\eqref{eq:psi-def} satisfy the recursions
\[
  \psi^{\pa}_{n+1} (x) \, = \, \psi^{\pa}_{n} (2 x) + \psi (x)
  \qquad \text{and} \qquad
  \psi^{\pa}_{n+1} (x) \, = \, \psi^{\pa}_{n} (x) + \psi (2^{n}x)
\]
as well as the symmetry relations\/
$\psi^{\pa}_{n} (1-x) = \psi^{\pa}_{n} (x)$.  \qed
\end{fact}

\begin{figure}[ht]
\includegraphics[width= 0.7\textwidth]{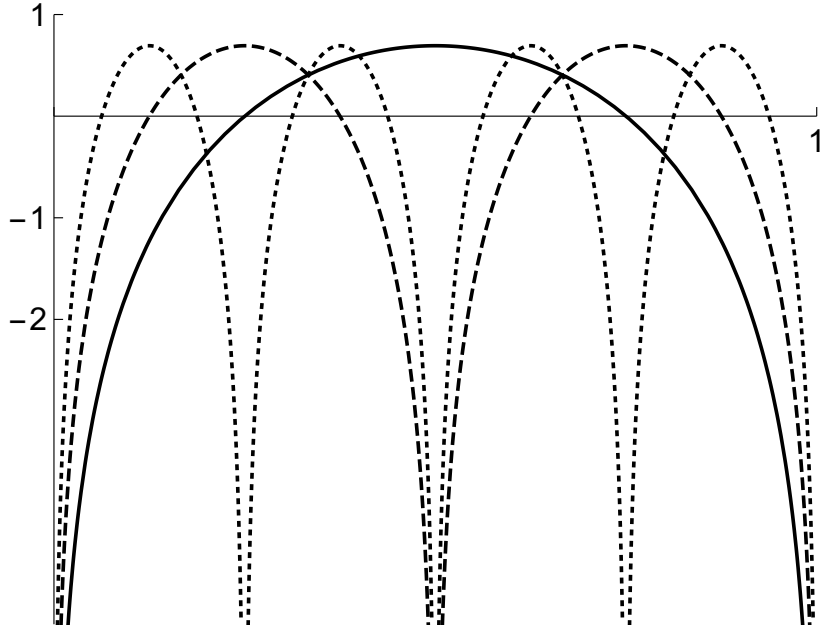}
\caption{\label{fig:fun}Illustration of the graphs of $\psi(x)$ (solid
  line), $\psi(2x)$ (dashed line) and $\psi(4x)$ (dotted line).}
\end{figure}

Observe that $\psi (2^{m} x)$ is a periodic function, with fundamental
period $2^{-m}$, which means that $\psi ( 2^{m} x)$ consists of $2^m$
`humps' of identical shape; see Figure~\ref{fig:fun} for an
illustration. Moreover, if $m\geqslant 2$, it is symmetric under a
reflection in $x=\frac{1}{2}$, with two humps on any cylinder of the
form $\langle \ts q^{\pa}_{1} \! \cdots q^{\pa}_{m-1} \rangle$. A
simple calculation then yields the following property, which is
illustrated in Figure~\ref{fig:birk-sum}.

\begin{fact}
  For any integer\/ $n\geqslant 2$, the function\/ $\psi^{\pa}_{n}$ is
  strictly concave on every cylinder of the form\/
  $\langle \ts q^{\pa}_{1} \! \cdots q^{\pa}_{n-1} \rangle$ with\/
  $q^{\pa}_{i} \in \{0,1\}$, and has singularities on the boundary
  points of them. In other words, $\psi^{\pa}_{n}$ consists of\/
  $2^{n-1}$ humps on\/ $[0,1]$.  \qed
\end{fact}

\begin{figure}[ht]
\includegraphics[width = 0.7\textwidth]{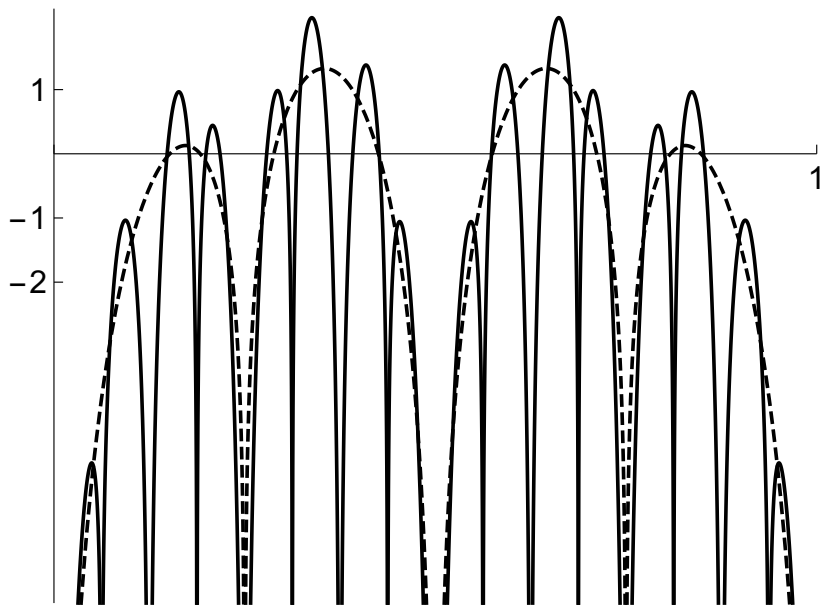}
\caption{\label{fig:birk-sum}Illustration of $\psi^{}_{3} (x)$ (dotted
  line) and $\psi^{}_{5} (x)$ (solid line).}
\end{figure}

\section{Maximal scaling exponent}\label{Sec:Max-exponent}

This section is devoted to the proof of the following proposition,  
which is crucial also for our main theorem. 

\begin{prop}\label{maxBeta}
  The maximal value of\/ $\beta$ is 
   given by
\[
  \max_{x\in [0,1]}\limsup_{n\to\infty}
  \frac{\psi^{\pa}_{n} (x)}{n \ts \log (2)}
  \, = \,  \beta \Bigl( \myfrac{1}{3} \Bigr)
  \, = \, \beta \Bigl( \myfrac{2}{3} \Bigr)
  \, = \,  \frac{\log (3/2)}{\log (2)} 
  \, \approx \, 0.584 {\,} 963 \ts .
\]    
\end{prop}

To prove this statement, we need to identify all cylinders
$\langle q^{\pa}_{1} \! \cdots q^{\pa}_{n-1}\rangle$ where the maximal
value of $\psi^{\pa}_{n}$ occurs.  Since it will turn out that such a
maximum never lies on one of the endpoints of such a cylinder, we can
go one step further and consider the cylinders
$\langle q^{\pa}_{1} \! \cdots q^{\pa}_{n} \rangle$ for
$\psi^{\pa}_{n}$. From Fact~\ref{fact:psi-props}, we know that
$\psi^{\pa}_{n} (\widehat{x}\ts ) = \psi^{\pa}_{n} (x)$ holds for all
$n\in\NN$. It thus suffices to consider
$x\in \bigl[ 0, \frac{1}{2} \bigr]$.

To continue, we introduce $\psi^{(n)}(x) := \psi(2^{n-1} x)$ for
$n \in \NN$, so that $\psi^{(1)} = \psi$ and
\[
     \psi^{\pa}_n (x) \, = \, 
    \psi^{(n)}(x) + \psi^{(n-1)}(x) + \cdots + 
    \psi^{(1)}(x) \ts .
\]
Furthermore, we need to go into some detail on the prefixes of a given
sequence.  Here, if $x=x^{\pa}_{1} x^{\pa}_{2} \cdots$ and $m\in\NN$,
we use the shorthand
$x^{[m]} := x^{\pa}_{1} x^{\pa}_{2} \nts \cdots x^{\pa}_{m}$ for the
prefix of $x$ of length $m$. Let us also fix a notation for the
alternating sequences in $\XX$, namely
\begin{equation}\label{eq:def-y-hat-y}
  y \, = \, \overline{01}\qquad\text{and}\qquad
     \widehat{y}\,=\, \overline{10} \ts ,
\end{equation}
where one implies the other.

\begin{lemma}\label{lem:prev-est}
  Let\/ $y \in \XX $ be the sequence from
  \eqref{eq:def-y-hat-y}. Then, the estimate
\begin{equation}\label{eq:ineq}
    \bigl(\psi + \psi^{(2)} \bigr) (y^{[2k]}0\ts x) \, > \,
     \bigl(\psi + \psi^{(2)}\bigr) (y^{[2k]}1\widehat{x} \ts )
\end{equation}
holds for any\/ $k \in \NN$ and all\/ $x\in\XX$ with\/
 $x \ne  \overline{1} := 111 \cdots$.
\end{lemma}

\begin{proof}
  Using the identities
  $\psi^{(2)}(y^{[2k]}0\ts x) = \psi (\widehat{y}^{\, [2k-1]}0\ts x) =
  \psi(y^{[2k-1]}1\widehat{x}\ts )$ in conjunction with
  $\psi^{(2)}(y^{[2k]}1\widehat{x} \ts ) = \psi (\widehat{y}^{\,
    [2k-1]}1\widehat{x} \ts) = \psi(y^{[2k-1]}0\ts x)$, we can rewrite
  Eq.~\eqref{eq:ineq} equivalently as
\begin{equation}\label{eq:prev-est-1}
     \psi(y^{[2k-1]}1\widehat{x} \ts ) - \psi(y^{[2k-1]}0\ts x) \, > \,
      \psi(y^{[2k]}1 \widehat{x} \ts ) - \psi(y^{[2k]}0\ts x ) \ts .
\end{equation}
Clearly, all arguments are in $ \bigl[ 0, \frac{1}{2} \bigr]$. On this
interval, $\psi$ is strictly increasing, while $\psi^{\ts\prime}$ is
positive and decreasing. Note also that
$y^{[2k-1]}0\ts x < y^{[2k]} 0\ts x $ and
$ y^{[2k-1]}1\widehat{x} \leqslant y^{[2k]}1 \widehat{x}$, and that we
have the inequality
\[
    0 \, < \, y^{[2k]}1 \widehat{x}  - y^{[2k]}0\ts x \, = \,
     2^{-2k} (1\widehat{x} - 0\ts x) \, < \, 
    2^{-2k+1} (1 \widehat{x} - 0\ts x) \, = \, y^{[2k-1]}1 
    \widehat{x} - y^{[2k-1]} 0\ts x  \ts .
\]
Thus, we get
\begin{align*}
    \psi(y^{[2k-1]}1\widehat{x} \ts ) - \psi(y^{[2k-1]}0\ts x) \,  & = 
     \int_{y^{[2k-1]}0\ts x}^{y^{[2k-1]}1\widehat{x}} 
     \psi^{\ts\prime}(u) \dd u \, > \, \int_{y^{[2k-1]}1\widehat{x} - 
        (y^{[2k]}1\widehat{x} - y^{[2k]}0\ts x)}^{y^{[2k-1]}1
           \widehat{x}} \psi^{\ts\prime}(u) \dd u   \\[2mm] 
           & \geqslant \int_{y^{[2k]}0\ts x}^{y^{[2k]}1\widehat{x}}
            \psi^{\ts\prime}(v) \dd v  \:  = \: \psi(y^{[2k]}1\widehat{x}) - 
            \psi(y^{[2k]}0\ts x) \ts , 
\end{align*}
where the first inequality is due to a reduction of the integration
region and the second follows via the substitution
$v = u + (y^{[2k]}1\widehat{x} - y^{[2k-1]}1\widehat{x} \ts )
\geqslant u $ so that $\psi^{\ts\prime}(u) \geqslant \psi^{\ts\prime}(v) $.
\end{proof}

This observation has the following consequence.

\begin{coro}\label{Coro:alternating}
  Suppose that\/ $y$ is the alternating sequence from
  \eqref{eq:def-y-hat-y}.  Then, for any\/ $n \in \NN$, one has\/
  $\, \psi^{\pa}_{n} \bigl( y^{[n+1]}x \bigr) > \psi^{\pa}_n \bigl(
  y^{[n]}\widehat{y}_{n+1} \widehat{x} \ts \bigr)$, provided that\/
  $x \ne \overline{a}:= aaa\ldots$ with\/ $a = y^{\pa}_{n+2}$.
\end{coro}

\begin{proof}
  We show the statement for odd and even $n$ separately. First, assume
  $n = 2m$ with $m \in \NN$. In this case, Lemma~\ref{lem:prev-est}
  implies
\begin{align}
    \psi^{\pa}_{2m} \bigl(y^{[2m]}0\ts x \bigr) \, & = \,
    \left( \bigl( \psi + \psi^{(2)} \bigr) 
     + \bigl( \psi^{(3)} + \psi^{(4)} \bigr)
     + \ldots + \bigl( \psi^{(2m-1)} 
     + \psi^{(2m)} \bigr) \right) \bigl( y^{[2m]}0\ts x \bigr)
      \notag \\
   &= \, \bigl( \psi + \psi^{(2)} \bigr) \bigl( y^{[2m]}0\ts x \bigr) 
    + \bigl( \psi + \psi^{(2)} \bigr) \bigl( y^{[2m-2]}0\ts x \bigr)
    + \ldots + \bigl( \psi + \psi^{(2)} \bigr) 
      \bigl( y^{[2]}0\ts x \bigr) \notag \\[1mm] 
    & > \, \bigl( \psi + \psi^{(2)} \bigr) 
       \bigl( y^{[2m]}1\widehat{x}\ts \bigr) 
     + \bigl( \psi + \psi^{(2)} \bigr) \bigl( 
     y^{[2m-2]}1 \widehat{x} \ts \bigr)
     + \ldots + \bigl( \psi + \psi^{(2)} \bigr)
      \bigl( y^{[2]}1 \widehat{x} \ts \bigr)
      \label{eq:psi2m} \\[1mm]
    & = \, \psi^{\pa}_{2m} \bigl( y^{[2m]}1 \widehat{x}
     \ts \bigr) . \notag
\end{align}
From this result, we can also proceed to $n = 2m+1$ with
$m \in \NN$,
\begin{align*}
   \psi^{\pa}_{2m+1} \bigl( y^{[2m+1]}1x \bigr) \, & = \,
   \psi^{\pa}_{2m} \bigl( \widehat{y}^{[2m]}1x \bigr) +
           \psi \bigl( y^{[2m+1]}1x \bigr) \, \geqslant \,
           \psi^{\pa}_{2m} \bigl( y^{[2m]}0\widehat{x} \ts \bigr)
       + \psi \bigl( y^{[2m+1]}0\ts \widehat{x} \ts \bigr) \\[1mm] 
       & > \, \psi^{\pa}_{2m} \bigl( y^{[2m]}1x \bigr) 
       + \psi \bigl( y^{[2m+1]}0\ts \widehat{x} \ts \bigr) \, = \,
        \psi^{\pa}_{2m} \bigl( \widehat{y}^{[2m]}0\widehat{x} \ts \bigr)
         + \psi \bigl( y^{[2m+1]}0\ts \widehat{x} \ts \bigr) \\[1mm]
       & = \, \psi^{\pa}_{2m+1} \bigl( y^{[2m+1]}0\widehat{x} \ts \bigr) .
\end{align*}
For the first inequality, we have used that $\psi$ is increasing on
$\bigl[ 0, \frac{1}{2} \bigr]$, while the second inequality uses the
corresponding identity for even indices. When $n=1$, the statement
directly follows from the strictly increasing nature of $\psi$ on
$\bigl[ 0, \frac{1}{2} \bigr]$.
\end{proof}

In order to narrow down the position of the maximum of
$\psi^{\pa}_{n}$, we would like to compare the values of
$\psi^{\pa}_n$ on cylinders of the form
$\langle q^{\pa}_1 \! \cdots q^{\pa}_n \rangle$, where we may choose
$q^{\pa}_{1} = 0$ due to the symmetry of $\psi^{\pa}_{n}$.  To this
end, we partition the interval $ \bigl[ 0, \frac{1}{2} \bigr]$ into
sets which contain exactly one element of each cylinder. We will show
that it is possible to choose these sets in such a way that the
maximum of $\psi^{\pa}_n$ on each such set lies \emph{always} in the
cylinder $\langle y^{\pa}_1 \! \cdots y^{\pa}_n \rangle$. The sets can
be constructed via an iterative reflection at the midpoints of
appropriately chosen cylinders. Let us make this more precise as
follows.

\begin{prop}\label{prop:pointwise-compare}
  Let\/ $y$ be the alternating sequence from \eqref{eq:def-y-hat-y}
  and denote the involution on\/ $\XX$ by\/ $I$, so\/
  $I(x) = \widehat{x}$, with\/ $I^{\ts 0} = \mathrm{id}$ as usual.
  Further, define
\[
   A_n (x) \, := \,  \big\{ \ts q^{\pa}_1 \! \cdots q^{\pa}_n 
   \ts I^{q^{\pa}_n}(x) \, : \,  q^{\pa}_1 = 0  \text{ and } \,
   q^{\pa}_2, \ldots,  q^{\pa}_n \in \{0,1\} \big\},
\]
for\/ $n\in\NN$ and\/ $x\in\XX$.  Then, one has\/
$ \bigl[ 0, \frac{1}{2} \bigr] = \bigcup_{x\in\XX} A_n (x)$.

Moreover, for all\/ $n \in\NN$ and\/ $x\in\XX$, the maximum of\/
$\psi^{\pa}_n$ on\/ $A_n (x)$ is taken at\/ $y^{[n]} I^{y_n} (x)$.
This maximum is strict as long as\/ $x\neq \overline{0}$, whereas it
is given by\/ $-\infty$ for\/ $x =\overline{0}$.
\end{prop} 

\begin{proof}
  The first claim is obvious, and the second can be shown by
  induction.  For $n=1$, one has $A_1 (x) = \{ 0\ts x \}$, which is a
  singleton set, and the assertion is trivial.  Suppose it is true up
  to $n$. Now, observe that
\[
     A_{n+1} (x) \, = \, A_n (0\ts x) \cup A_n (1\widehat{x}\ts ) \ts .
\]
By the induction assumption, $\psi^{\pa}_n$ takes its maximum on
$A_n (0\ts x)$ in the point
 \[
    y^{[n]} I^{y^{\pa}_n} (0\ts x) \, = \,
    y^{[n]} y_n I^{y^{\pa}_n}(x) \, = \, 
    y^{[n]} \widehat{y}^{\pa}_{n+1} I^{y^{\pa}_n} (x) \ts .
 \]
 Similarly, the position of the maximum on $A_n (1\widehat{x}\ts )$ is
 given by
 $y^{[n]} I^{y^{\pa}_n} (1 \widehat{x}\ts ) = y^{[n+1]}
 I^{y^{\pa}_{n}} (\widehat{x}\ts )$.  Comparing these two points, we
 obtain that $\psi^{\pa}_n$ takes its maximum on $A_{n+1} (x)$ at the
 position $y^{[n+1]} I^{y^{\pa}_{n+1}}(x) $, by an application of
 Corollary~\ref{Coro:alternating} (note that $\psi^{\pa}_n$ is
 strictly larger than $-\infty$ at that point).  The proof is
 completed by the observation that
 $\psi^{\pa}_{n+1} - \psi^{\pa}_{n} = \psi^{(n+1)}$ is constant on
 $A_{n+1} (x)$, because
\[
    \psi^{(n+1)} \bigl(q^{\pa}_1 \! \cdots q^{\pa}_{n+1} 
      I^{q^{\pa}_{n+1}}(x) \bigr) 
    \, = \, \psi \bigl( q^{\pa}_{n+1} I^{q^{\pa}_{n+1}}(x) \bigr)
       \, = \, \begin{cases}
    \psi (0\ts x), & \text{for } \, q^{\pa}_{n+1} = 0 \ts , \\
    \psi (1\widehat{x}\ts ), & \text{for } \, q^{\pa}_{n+1} = 1 \ts ,
   \end{cases}
\]
where $\psi (1 \widehat{x}\ts ) = \psi (0 \ts x)$ by the reflection
symmetry of $\psi$ on $[0,1]$. Obviously, $\psi(0\ts x) = - \infty$ if
and only if $x = \overline{0}$.
\end{proof}

\begin{remark}
  In Proposition~\ref{prop:pointwise-compare}, we have singled out the
  points with $x = \overline{0}$, because these are the positions
  where divergences of $\psi^{\pa}_n$ occur. At first sight, it seems
  natural to also consider the special case $x = \overline{1}$, which
  corresponds to the opposite boundary points of the cylinders of the
  form $\langle q^{\pa}_1 \! \cdots q^{\pa}_n \rangle$. However, these
  points are just the midpoints of the cylinders of type
  $\langle q^{\pa}_1 \! \cdots q^{\pa}_{n-1} \rangle$. Although the
  cardinality of the set $A_n (x)$ gets reduced by a factor of $2$ in
  that case, the argument we employed above remains unaltered.

  Let us also mention that the statement in
  Proposition~\ref{prop:pointwise-compare} can slightly be sharpened
  to hold also for all $\psi^{\pa}_m$, with $m > n$, which take the
  maximum on $A_n (x)$ in the same point as $\psi^{\pa}_n$. The value
  of the maximum is $-\infty$ if and only if
  $q^{\pa}_1 \! \cdots q^{\pa}_n I_{\vphantom{u}}^{q^{\pa}_n} (x)$ is
  a singularity point for $\psi^{\pa}_m$ for some (equivalently every)
  choice of $q^{\pa}_1 \! \cdots q^{\pa}_n$. This follows from the
  fact that $\psi^{(m)}$ is constant on $A_n (x)$ for all
  $m \geqslant n$.  \exend
\end{remark}

\begin{coro}\label{coro:maxima}
  For any\/ $n \in \NN$, the function\/ $\psi^{\pa}_n$ takes its
  maximal value on\/ $ \bigl[ 0, \frac{1}{2} \bigr]$ in the cylinder\/
  $C_n = \langle y^{\pa}_1 \! \cdots y^{\pa}_n \rangle$, with
  $y = \overline{01}$ as in \eqref{eq:def-y-hat-y}, while the maximal
  value on\/ $ \bigl[ \frac{1}{2}, 1 \bigr]$ is taken in\/ $ I (C_n)$,
  and these two maximal values are equal. Moreover, in terms of
  cylinders that are coarser by one level, the two maxima of\/
  $\psi^{\pa}_{n}$ lie in the interior of\/ $C_{n-1}$ and\/
  $I(C_{n-1})$.
\end{coro} 

\begin{proof}
  Clearly, $\psi^{\pa}_n$ has a maximum on the cylinder
  $\langle y^{\pa}_1 \! \cdots y^{\pa}_n \rangle$, say in the point
  $x^{\star}$. Since the only singularity in this cylinder is at
  $y^{[n-1]}\ts \overline{y_n}$, we know that
  $\psi^{\pa}_n (x^{\star}) > - \infty$.
   
  Let $z \in \langle q^{\pa}_1 \! \cdots q^{\pa}_n \rangle$ with
  $q^{\pa}_{1} = 0$, so $z\in \bigl[ 0, \frac{1}{2} \bigr]$. Then, by
  Proposition~\ref{prop:pointwise-compare}, $z \in A_n (x)$ for some
  $x \in \XX$, and we have
\[
    \psi^{\pa}_n (x^{\star}) \, \geqslant \, \psi^{\pa}_n
    \bigl( y^{[n]} I^{y^{\pa}_n} (x) \bigr) 
    \, \geqslant \, \psi^{\pa}_n (z) \ts .
\]
Unless $z=x^{\star}$, at least one of the inequalities is strict, which
proves the first claim.  The second is an obvious consequence of the
symmetry under $I$.

For the last statement, we know from
Proposition~\ref{prop:pointwise-compare} that the position of the
maximum, $x^{\star}$, is not at the boundary point of $C_n$ with the
singularity.  Consequently, it is either an interior point or the
other boundary point, and hence an interior point of the coarser
cylinder, $C_{n-1}$. The mirror statement holds for $I(C_n)$, and we
are done.
\end{proof}

\begin{lemma}\label{lem:const}
  Let\/ $y$ be as in Eq.~\eqref{eq:def-y-hat-y}. Then, there exists a
  constant\/ $K > 0$\/ such that the inequality\/
  $\, \max^{\pa}_{x \in [0,1]} \psi^{\pa}_n(x) - \psi^{\pa}_n ( y )
  \leqslant K$\/ holds for all\/ $n \in \NN$.
\end{lemma}

\begin{proof}
  Due to symmetry, it suffices to consider the maximum on the interval
  $ \bigl[ 0, \frac{1}{2} \bigr]$. By Corollary~\ref{coro:maxima}, we
  obtain
\begin{align*}
  \max_{x \in [0,1]} \bigl( \psi^{\pa}_{n+1}(x) -
      \psi^{\pa}_{n+1}  (y) \bigr)
  \, & =  \max_{x \in C_{n+1}} \bigl( \psi^{\pa}_{n+1}(x)
       - \psi^{\pa}_{n+1} (y) \bigr)  \\[1mm]       &\leqslant 
   \max_{x \in C_{n+1}} \bigl( \psi^{\pa}_{n}(2x) - \psi^{\pa}_n 
   (   \widehat{y} \ts ) \bigr) \, + 
   \max_{x \in C_{n+1}} \bigl( \psi (x) - \psi  (y) \bigr) .
\end{align*}
Note that
$\max^{\pa}_{x \in C_{n+1}} \psi^{\pa}_n (2x) = \max^{\pa}_{x \in
  I(C_n)} \psi^{\pa}_n (x) = \max^{\pa}_{x \in C_n} \psi^{\pa}_n (x)$
by the definition of the cylinders.  Since $y = \overline{01} \in C_n$
for all $n \in \NN$, we infer from the mean value theorem that
\[
   \max_{x \in C_{n+1}} \bigl(\psi (x) - \psi  (y)
   \bigr) \, \leqslant \,  \frac{\left| \psi^{\ts\prime}
       (\xi) \right|}{2^{n+1}} \ts  ,
\]
with some $\xi \in C_{n+1}$. For any $n \in \NN$, one has
$C_{n+1} \subset \bigl[ \frac{1}{4}, \frac{1}{2} \bigr]$, and since
$\left| \psi^{\ts\prime} \right|$ is decreasing on this interval,
$\left| \psi^{\ts\prime} (\xi) \right| \leqslant \psi^{\ts\prime}
\left( \frac{1}{4} \right) = 2 \pi$. Consequently,
\[
  \max_{x \in [0,1]} \bigl( \psi^{\pa}_{n+1} (x) -
     \psi^{\pa}_{n+1}  (y) \bigr) \, \leqslant \,
     \max_{x \in C_n} \bigl( \psi^{\pa}_n (x) - \psi^{\pa}_n
    (y) \bigr) + \myfrac{\pi}{2^{n}} \ts .
\]
Recursively, we then obtain
\[
    \max_{x \in [0,1]} \bigl(\psi^{\pa}_{n+1} (x)
    - \psi^{\pa}_{n+1}  (y)\bigr) \, \leqslant \,
    \max_{x \in C^{\pa}_1} \bigl( \psi(x) - \psi  (y)
    \bigr) \, + \sum_{k = 1}^n 
    \myfrac{\pi}{2^k} \, \leqslant  \,
    \pi + \log \Bigl( \myfrac{4}{3} \Bigr) ,
\]
where we have used $\psi (y) = \log \bigl( \frac{3}{2} \bigr)$ and
$\max^{\pa}_{x \in C^{\pa}_1} \psi(x) = \psi \bigl( \frac{1}{2}
\bigr) = \log(2)$.
\end{proof}

\begin{proof}[Proof of Proposition~\textnormal{\ref{maxBeta}}]
  Let $n \geqslant 2$.  From Corollary~\ref{coro:maxima}, we know that
  the maximum of $\psi^{\pa}_{n}$ is taken in the interior of the
  cylinders\/ $\langle 0101 \cdots 01\rangle$ and\/
  $\langle 1010 \cdots 10\rangle$ when\/ $n$ is even, and in the
  cylinders\/ $\langle 0101 \cdots 010\rangle$ and\/
  $\langle 1010 \cdots 101\rangle$ when\/ $n$ is odd.  Since
  $\frac{1}{3} = \overline{01} = y$ and
  $\frac{2}{3}= \overline{10}=\widehat{y}$, see
  Eq.~\eqref{eq:def-y-hat-y}, the claim on the location follows from
  Lemma~\ref{lem:const}, with the value given by a simple calculation
  as in \cite{BGN}.
\end{proof}

\begin{remark}
  Both Proposition~\ref{maxBeta} and Lemma~\ref{lem:const} (with an
  improved constant $K$) also follow from bounds on
  $\| P^{\pa}_n \|^{\pa}_{\infty}$ with $n \in \NN$ that were
  established in \cite{gelfond}; see also \cite[Thm.~1.1]{Q1}.  \exend
\end{remark}

\section{Gibbs-type properties}\label{Sec:Gibbs-type}

So far, we have shown that the level sets $\cB(\alpha)$ are indeed
empty for $\alpha > \log \bigl( \frac{3}{2} \bigr)$. As a next step
towards the proof of Theorem~\ref{thm:MAIN-MF-1}, we will establish a
link between the local dimension of $\nu$ and the Birkhoff average of
$\psi$ at certain points $x \in \XX$. Since our arguments will evolve
along similar lines, let us first sketch how the corresponding
relation arises for Gibbs measures.

For any H\"older continuous potential $\phi$ on
$(\XX,\varrho^{\pa}_2)$, there is a unique $T$-invariant Borel
probability measure $\mu$ that satisfies
\begin{equation}\label{eq:Gibbs-measure}
  c^{\pa}_1 \, \leqslant\, \frac{\mu(\langle x^{\pa}_1
    \! \cdots x^{\pa}_n \rangle)}{\exp\left(-P n +
      \phi_n(x)\right)}  \, \leqslant \, c^{\pa}_2 \ts ,
\end{equation}
for any $x \in \XX$, $n \in \NN$ and some constants
$c^{\pa}_1 , c^{\pa}_2 > 0$ and $P \in \RR$. Here,
$\phi^{\pa}_n = \sum_{\ell=0}^{n-1} \phi \circ \sigma^n$, in analogy
to $\psi^{\pa}_n$, and $P$ turns out to be the topological pressure of
$\phi$; compare Section~\ref{Sec:Eq-measure}. Following the
terminology of \cite{Bow}, we call $\mu$ an invariant \emph{Gibbs
  measure} for $\phi$. The property in \eqref{eq:Gibbs-measure}
immediately allows to conclude that
\begin{equation}\label{eq:Gibbs-dim-birk-connection}
  \lim_{n \rightarrow \infty} \frac{\phi^{\pa}_n(x)}{n}
  \, = \, P - \dim_{\mu} (x)  \log(2) 
\end{equation}
holds for all $x \in \XX$, see \cite[Prop.~1]{PesinWeiss}, thus
establishing a connection between the dimension spectrum and the
Birkhoff spectrum at the same time.

Due to the singularity of $\psi$ at $0$, the function $\psi^{\pa}_n$
has infinite variation on any cylinder of the form
$\langle x^{\pa}_1 \! \cdots x^{\pa}_n \rangle$, prohibiting the
analogue of \eqref{eq:Gibbs-measure} for \emph{any} measure on
$\XX$. The aim of this section is to establish a slightly weaker but
similar relation for $\nu$ that suffices to derive a relation that is
analogous to \eqref{eq:Gibbs-dim-birk-connection}, at least for points
$x$ in certain subshifts of $\XX$.

Given $m \in \NN$, restricting our space to the \emph{subshift of
  finite type} (SFT)
\[
  \XX_m \, := \, \left\{x\in\XX \, : \,  x^{\pa}_{\ell}
    \! \cdots x^{\pa}_{\ell+m} \notin \{ 0^{m+1}, 1^{m+1} \} ,
    \, \ell\in \NN \right\} 
\]
ensures that all $x \in \XX_m$ are bounded away from the singularity
points of $\psi$ by at least $2^{-m}$.  We define the set of
admissible words of length $n>m$ in $\XX_m$ as
\[
  \Sigma_m^n \, := \, \left\{\omega\in \{0, 1 \}^n \, : \,
    \omega^{\pa}_{\ell} \! \cdots  \omega^{\pa}_{\ell+m}
    \notin \{0^{m+1}, 1^{m+1} \} , \,
    \ell\in \{ 1,\ldots, n\nts - \nts m \} \right\},
\]
and use $\Sigma^n = \{0,1\}^n$ for the set of all binary words of
length $n$. One can verify that the restriction of $\psi$ to $\XX_m$
is indeed H\"older continuous, and estimate its modulus of continuity.

\begin{lemma}\label{lem:deriv}
  For any\/ $x\in [0,1 ]$, we have
\[
  \left|\psi^{\ts\prime}(x)\right|
  \, \leqslant \, 2\ts \max\left\{\myfrac{1}{x},
    \myfrac{1}{1-x}\right\}
    \]
    with respect to the Euclidean metric\/
    $\varrho^{\pa}_{\mathrm{E}}$ on\/ $[0,1]$.  Moreover, $\psi$ is
    H\"older continuous on\/ $\XX_m$ with respect to the metric\/
    $\varrho^{\pa}_{2}$ on the shift space\/ $\XX$.
  \end{lemma}

\begin{proof}
Using \eqref{eq:psi-deriv}, we obtain 
\begin{equation}\label{eq:psi-x}
  \lim_{x\mbox{\tiny $\searrow$} 0} \, x \,
    \psi^{\ts\prime}(x)
    \, = \, \lim_{x\mbox{\tiny $\searrow$} 0}
    \frac{2\pi x    \bigl(1+\cos (2\pi x )\bigr)}
    {\sin (2\pi x)}  \, = \, 2
\end{equation}
and 
\[
  \myfrac{\dd}{\dd x} \left( x \, \psi^{\ts\prime}(x)
    \right) \,=\, \frac{2\pi\sin (2\pi x)
    -4\ts\pi^2 x}{1-\cos (2\pi x)} \ts .
\]
Since $\sin (x) < x$ for $x>0$ and $1-\cos (2\pi x )>0$ on $(0,1)$, we
see that the derivative of $x \ts\ts \psi^{\ts\prime} (x)$ is
negative, so $x \ts\ts \psi^{\ts\prime} (x)$ is monotonically
decreasing on $\left(0, \frac{1}{2} \right]$.  Combining this with
Eq.~\eqref{eq:psi-x} gives $\psi^{\ts\prime}(x)\leqslant \frac{2}{x}$. 
The estimate $\left|\psi^{\ts\prime}(x)\right|\leqslant \frac{2}{1-x}$
follows from the symmetry of $\psi$.

To prove the second claim, we note that, for all $m,n\in\NN$ with
$n>m$, we have
\[
    \sup_{\omega\in\Sigma_{m}^n}\,
    \sup_{x,y\in\langle\omega\rangle\cap\XX_m}
   \left|\psi(x)-\psi(y)\right|
   \, \leqslant  \sup_{\omega\in\Sigma_{m}^n} \,
   \sup_{x\in\langle\omega\rangle\cap\XX_m}
   \left|\psi^{\ts\prime}(x)
   \right|\, \left|\langle\omega\rangle\right|
   \,=  \sup_{x\in\XX_m} \frac{\left|\psi^{\ts\prime}
     (x)\right|}{2^{n}} \ts .
\]
By the concavity of $\psi$, see Fact~\ref{fact:psi-props} and the
first statement of the lemma, we have the estimate
$\sup_{x\in\XX_m}\left|\psi^{\ts\prime}(x) \right|<\psi^{\ts\prime}
(2^{-m-1})\leqslant 2^{m+2}$.  Consequently,
\[
  \sup_{\omega\in\Sigma_{m}^n} \,
  \sup_{x,y\in\langle\omega\rangle\cap\XX_m}
  \left|\psi(x)-\psi(y)\right|
  \, \leqslant \, 2^{m+2-n}
\]
and hence
$\left|\psi(x)-\psi(y)\right| \leqslant 2^{m+2}\ts \varrho^{\pa}_2
(x,y)$.  Thus, $\psi$ is Lipschitz continuous with Lipschitz constant
$2^{m+2}$ on $\XX_m$.
\end{proof}

For $x \in \XX$, denote by
$C_n(x) = \langle x^{\pa}_1 \! \cdots x^{\pa}_n\rangle$ the (unique)
cylinder of length $2^{-n}$ that contains $x$. We are concerned with
the values of $\psi^{\pa}_n$ on such cylinders. Recall that
$\psi^{\pa}_{n}(x)$ comprises $2^{n-1}$ humps of the same total width
so that $\psi^{\pa}_{n}$ is concave on $C_n(x)$, singular at one
boundary point, and non-singular at the other; compare
Figure~\ref{fig:birk-sum}. Taking the intersection of $C_n(x)$ with
$\XX_m$ removes a neighbourhood around the singularity, and we find
that the variation of $\psi^{\pa}_{n}$ on such a set is bounded in a
suitable way.

\begin{lemma}\label{Lemma:SFT-sup-inf-bounds}
  Let\/ $x \in \XX_m$ for some\/ $m \in \NN$. Then, there exists a
  constant\/ $K = K(m) >0$ such that, for all\/ $n \in \NN$,
\[
   \sup_{y \in C_n(x)\cap\XX_m} \! \exp(\psi^{\pa}_{n}(y))
   \, \leqslant \, K \! \inf_{y \in C_n(x)\cap\XX_m}
   \! \exp(\psi^{\pa}_{n}(y)) \ts .
\]
\end{lemma}

\begin{proof}
  The H\"older continuity of $\psi$ on $\XX_m$, see Lemma
 ~\ref{lem:deriv}, implies that there exists a bounded distortion
  constant $W>0$ such that
\[
   \sup_{y  \in C_n(x)\cap\XX_m} \!\!
   \psi^{\pa}_{n}(y) \; - \! \inf_{y \in C_n(x)\cap\XX_m}
   \! \! \psi^{\pa}_{n}(y) \, \leqslant \, W \ts .
\]
Setting $K=\exp(W)$ gives the claim.
\end{proof}

Using the self-similarity properties of $\nu$, we can relate its
values on cylinders of the form $C_n(x)$ with extreme values of the
level-$n$ approximants $P_n$. A non-trivial lower estimate is
available for points $x \in \XX_m$ as follows.

\begin{lemma}\label{Lemma:bounds-nu}
  If\/ $x \in \XX$, we have the following bound for the value of\/
  $\nu$ on\/ $C_n(x)$,
\begin{equation}\label{Eq:nu-upper-bound}
  \nu(C_n(x)) \, \leqslant \, {2^{-n}} \!
  \sup_{y \in C_n(x)} \exp(\psi^{\pa}_{n}(y)) \ts .
\end{equation}
Further, for\/ $x\in\XX_m$ and\/ $m\in\NN$, there exists a constant\/
$K'$, independently of\/ $n$, so that
\begin{equation}\label{Eq:nu-upper-lower-bound}
  \myfrac{1}{2^n \ts K'} \inf_{y \in C_n(x)\cap\XX_m}
  \exp(\psi^{\pa}_{n}(y))
  \, \leqslant \, \nu(C_n(x))  
  \, \leqslant \, \myfrac{K'}{2^{n}}
  \sup_{y \in C_n(x)\cap\XX_m} \exp(\psi^{\pa}_{n}(y)) \ts .
\end{equation}
\end{lemma}

\begin{proof}
  To establish \eqref{Eq:nu-upper-bound}, we note that,
  for $N > n$, one has
\begin{align*}
   \nu^{\pa}_N(C_n(x)) & \, =
   \int_{C_n(x)} P^{}_N(\xi) \dd \xi \, \leqslant 
   \sup_{y \in C_n(x)} P^{}_n(y) \int_{C_n(x)} P^{}_{N-n}(2^n \xi)
   \dd \xi \\[1mm]
    & \, =  \sup_{y \in C_n(x)} \exp(\psi^{\pa}_{n}(y))
    \; {2^{-n}}\! \int_0^1 P^{}_{N-n}(\xi) \dd \xi
    \, = \, {2^{-n}} \! \sup_{y \in C_n(x)} \!
    \exp(\psi^{\pa}_{n}(y)) \ts ,
\end{align*}
where we used the fact that $P^{\pa}_{N-n}$ is a probability
density on $[0,1]$. Taking $N \rightarrow \infty$ in the above
relation yields \eqref{Eq:nu-upper-bound}.

Let $C_n^0(x) = \langle x_1 \! \cdots x_n 0 \rangle$ and
$C_n^1(x) = \langle x_1 \! \cdots x_n 1 \rangle$ denote the left and
right half of this interval, respectively.  For
$j\in\left\{0,1\right\}$, we find
\begin{align*}
  \nu^{\pa}_N(C_n(x)) \, & \geqslant \int_{C_n^j(x)} P^{}_N(\xi) 
  \dd \xi \, \geqslant  \inf_{y \in C_n^j(x)}  \! P^{}_n(y)
  \int_{C_n^j(x)} P^{}_{N-n}(2^n \xi) \dd \xi \\[1mm]
   & = \inf_{y \in C_n^j(x)} \exp(\psi^{\pa}_{n}(y))
   \; {2^{-n}} \! \int_{j/2}^{(j+1)/2} P_{N-n}(\xi) \dd \xi 
     \, = \, {2^{-n-1}} \! \inf_{y \in C_n^j(x)} \!
     \exp(\psi^{\pa}_{n}(y)) \ts ,
\end{align*}
using that $P^{\pa}_{N-n}$ is symmetric under $x \mapsto 1-x$ in
$[0,1]$. Again, performing $N \rightarrow \infty$ gives
\[
  \nu(C_n(x)) \, \geqslant \, {2^{-n-1}}
  \max_{j\in\{0,1\}}\inf_{y \in C_n^j(x)} \!
  \exp(\psi^{\pa}_{n}(y)) \ts .
\]
By assumption, we have $m\geqslant 1$. To continue, for
$B\subset [0,1]$, define a $B$-truncated version of $\psi$ as
$\psi^B := \psi \ts\ts \mathbf{1}^{\pa}_{\nts\nts B}$.  With
$D(m) := [2^{-m-1}, 1-2^{-m-1}]$, consider the function
\begin{equation}\label{eq:def-psi-m}
  \psi^{D(m)} \, := \, \psi \, \mathbf{1}^{\pa}_{\nts D(m)}
   \, \geqslant \, \psi \ts ,
\end{equation}
viewed as a function on $\XX$. Clearly, $\psi^{D(m)}$ is 
H\"older continuous relative to the metric $\varrho^{\pa}_2$, 
and $\psi^{D(m)}(z) = \psi(z)$ as long as 
$z \neq \langle \omega \rangle$ for
$\omega \in \{ 0 \cdots 0 \ts , \nts 1 \cdots 1 \} \subset
\{0,1\}^{m+1}$. In particular, this holds for $z \in \XX_m$.  We
denote the bounded distortion constant of $\psi^{D(m)}$ by $W'$.

Now, choose $j=j(n)= \widehat{x}^{\pa}_n$.  Since $x\in\XX_m$, this
implies
\[
\begin{split}
    \inf_{y \in C_n^j(x)\cap \XX_m} \! \psi^{\pa}_{n}(y)
    \,  & -\inf_{y \in C_n^j(x)} \! \psi^{\pa}_{n}(y)
    \;  \leqslant \, \sum_{k = 0}^{n-1}\,
    \sup_{y \in C_n^j(x)\cap \XX_m} \!\! \psi(2^k y)\; -
    \!\! \inf_{y \in C_n^j(x)} \! \psi(2^k y) \\[1mm]
    & \leqslant \, \sum_{k=0}^{n-1}\,
    \sup_{y \in \langle x^{\pa}_{k+1} \cdots \, x^{\pa}_{n}
          \widehat{x}^{\pa}_{n} \rangle\cap \XX_m}
    \psi^{D(m)}(y) \; - \!\!
    \inf_{y \in \langle x^{\pa}_{k+1} \cdots \, x^{\pa}_{n}
         \widehat{x}^{\pa}_{n} \rangle}
    \! \psi^{D(m)}(y) \ts .
\end{split} 
\]
Assume $n \geqslant m$. Now, using H\"older continuity and setting
$K' = \ee^{W'}$ gives
\[
  \inf_{y \in C_n^j(x)\cap \XX_m} \! \psi^{\pa}_{n}(y)
  \; - \! \inf_{y \in C_n^j(x)} \! \psi^{\pa}_{n}(y)
  \, \leqslant \, W' \, = \,  \log( K') \ts .
\]
Since
$\inf_{y \in C_n^j(x)\cap \XX_m} \psi^{\pa}_{n}(y)\geqslant \inf_{y
  \in C_n(x)\cap \XX_m} \psi^{\pa}_{n}(y)$, we have established the
lower bound.

The upper bound follows by a similar calculation.  Using
$\psi^{\pa}_{n}(y) \leqslant \psi_{n}^{D(m)}(y)$ for $y \in \XX$, we find
\[
  \sup_{y \in C_n(x)} \! \psi^{\pa}_{n}(y) \; - \! \sup_{y \in C_n(x)
    \cap \XX_m}\! \psi^{\pa}_{n}(y) \, \leqslant \sup_{y \in C_n(x)} \!
  \psi^{D(m)}_n(y) \;  - \! \sup_{y \in C_n(x) \cap \XX_m}
  \! \psi^{D(m)}_n(y) < \log (K')
\]
as claimed.
\end{proof}

The following two results should be compared with
Eq.~\eqref{eq:Gibbs-dim-birk-connection}. We first consider only
cylinders as shrinking neighbourhoods of a point $x \in \XX_m$
before we allow for more general balls $B(x,r)$ with $r > 0$, which
may take different forms when built with respect to the Euclidean
metric.

\begin{prop}\label{Prop:measure-birkhoff-asymptotic-equivalence}
   Let\/ $m \in \NN$. For all\/ $x \in \XX_m$, one has
\[
  \lim_{n \rightarrow \infty} \myfrac{1}{n} \log \bigl( \nu(C_n(x))
  \bigr) \, = \ts \lim_{n \rightarrow \infty} \myfrac{1}{n} \log \left(
    {2^{-n}} \exp ( \psi^{\pa}_{n}(x)) \right) \, = \, - \log(2)
   \, + \lim_{n \rightarrow \infty} \frac{\psi^{\pa}_{n}(x)}{n} \ts ,
\]
provided that any of the limits exists.
\end{prop}

\begin{proof}
  By Lemmas~\ref{Lemma:SFT-sup-inf-bounds} and \ref{Lemma:bounds-nu},
  we have
\begin{equation}\label{Eq:meas-birkhoff-asymptotic-equiv}
\begin{split}
  \myfrac{1}{K K'} &\, \leqslant \, \frac{2^{-n} (K')^{-1} \inf_{y \in
      C_n(x)\cap \XX_m} \exp(\psi^{\pa}_{n}(y))}{2^{-n} \sup_{y \in
      C_n(x)\cap \XX_m} \exp(\psi^{\pa}_{n}(y))} \, \leqslant \,
  \frac{\nu(C_{n}(x))}{2^{-n} \exp(\psi^{\pa}_{n}(x))} \\[2mm]
  & \, \leqslant \, \frac{2^{-n} K' \sup_{y \in C_n(x)\cap \XX_m}
    \exp(\psi^{\pa}_{n}(y))}{2^{-n} \inf_{y \in C_n(x)\cap \XX_m}
    \exp(\psi^{\pa}_{n}(y))} \, \leqslant \, K K'.
\end{split}
\end{equation}
If any of the limits exists, we obtain our assertion by
\eqref{Eq:meas-birkhoff-asymptotic-equiv}.
\end{proof}

\begin{coro}\label{Coro:SFT-local-dim-birkhoff-relation}
  Consider\/ $\XX_m$ as a metric subspace of either\/
  $(\XX,\varrho^{\pa}_2)$ or\/
  $(\TT,\varrho^{\pa}_{\mathrm{E}})$. Then, for any\/
  $\tau \in\{\mathrm{E}, 2\}$ and\/ $x \in \XX_m$, one has
\[
   \locdim_{\nu,\tau}(x) \, = \, \lim_{r \rightarrow 0}
   \frac{\log \bigl(\nu(B_{\tau}(x,r))\bigr)}
   {\log(r)} \, = \, 1 - \myfrac{1}{\log(2)} \lim_{n
    \rightarrow \infty} \frac{\psi^{\pa}_{n}(x)}{n} \ts ,
\]
provided that any of the limits exists.
\end{coro}

\begin{proof}
  Consider the metric space $(\XX,\varrho^{\pa}_2)$ first.
  For $r < 1$, one has
\[
  B_2(x,r) \, = \, \langle x^{\pa}_1 \! \cdots x^{\pa}_M
  \rangle \, = \, C^{\pa}_{\nts M} (x) \ts ,
\]
where $M = M(r) = \lceil \log_{1/2}(r) \rceil$ and the claim is
immediate from
Proposition~\ref{Prop:measure-birkhoff-asymptotic-equivalence} since,
for any $(r_n)^{\pa}_{n \in \NN}$ with $r_n > 0$ and
$\lim_{n\to\infty} r_n = 0$, we have
\[
  \lim_{n \rightarrow \infty}
  \frac{\log \bigl(\nu(B_2(x,r_n))\bigr)}{\log(r_n)}
  \, =  \lim_{n \rightarrow \infty}
  \frac{\log \bigl(\nu(C_{\nts M(r_n)}(x))\bigr)}
  {-M(r_n) \log(2)} \, = \, 1 - \myfrac{1}{\log(2)}
  \lim_{n \rightarrow \infty} \frac{\psi^{\pa}_{n}(x)}{n} \ts .
\]

Next, we consider $( \TT,\varrho^{\pa}_{\mathrm{E}})$ and regard any
ball $B_2(x,r)$ as a subset of $ \TT$ with obvious meaning. Since, for
any $n \in \NN$, $B_2(x,2^{-n})$ has length $2^{-n}$ as an interval in
Euclidean space, it is
$B_2(x,2^{-n}) \subseteq B_{\mathrm{E}}(x,2^{-n})$. By the definition
of $\XX_m$, the Euclidean distance of $x$ from the boundary points of
the interval
$B_2(x,2^{-n}) = \langle x^{\pa}_1 \! \cdots x^{\pa}_n \rangle$ is at
least $2^{-n-m-1}$ for any $n \in \NN$. Thereby, for
$n \geqslant m+1$, we have
\[
  \frac{\,\log \bigl(\nu(B_2(x,2^{-n}))\bigr)}
      {\log(2^{-n})} \, \leqslant \,
  \frac{\,\log \bigl(\nu(B_{\mathrm{E}}(x,2^{-n}))\bigr)}
      {\log(2^{-n})} \, \leqslant \,
  \frac{\,\log \bigl(\nu(B_2(x,2^{-n+m+1}))\bigr)}
      {\log(2^{-n})},
\]
which gives the desired result as $n \rightarrow \infty$ for the
sequence with $r_n = 2^{-n}$. For general sequences
$(r_n)^{\pa}_{n \in \NN}$ that tend to $0$, the corresponding identity
follows by interpolation.
\end{proof}

\begin{remark}
  An immediate consequence of
  Corollary~\ref{Coro:SFT-local-dim-birkhoff-relation} and
  Proposition~\ref{maxBeta} is that, for $y = \overline{01}$ and
  $\tau \in \{\mathrm{E},2\}$, one has
\[
  \dim_{\nu,\tau}(y) \, = \, 2 - \frac{\log(3)}{\log(2)}
  \, \approx \, 0.415 \ts .
\]
On the other hand, Eq.~\eqref{Eq:nu-upper-bound} together with
Corollary~\ref{maxBeta} gives
$\dim_{\nu,\tau}(x) \geqslant 2 - \frac{\log (3)}{\log(2)}$ 
for all $x \in \XX$ and $\tau \in \{\mathrm{E},2\}$ by direct 
calculation. So, we actually get
$\inf_{x\in\XX} \dim_{\nu,\tau}(x)  = 2 - \frac{\log (3)}{\log(2)}$,
which verifies a conjecture from \cite[Sec.~4.4.2]{Q1}.  \exend
\end{remark}

\section{Restricted pressure function and the exhaustion
     principle}\label{Sec:Restricted-pressure}

Interpreting $\psi$ as a function on the symbolic space $\XX$, the
topological pressure from Eq.~\eqref{eq:def-pressure} for $t\ts \psi$
can be rewritten as
\begin{equation}\label{Eq:pressure-functional-psi}
  p(t) \,=\, \cP (t\ts \psi) \, = \,
  \lim_{n\to\infty}\myfrac{1}{n}\log\sum_{\omega\in\Sigma^n}
  \sup_{x\in \langle\omega\rangle}\exp \bigl(t\ts \psi^{\pa}_{n}(x)\bigr) ,
\end{equation}
where the limit exists by subadditivity. We shall see that the mapping
$p \colon \RR \xrightarrow{\quad} \RR \cup \{+\infty\}$ defines a
proper, convex function. We denote its \emph{Legendre transform} by
\begin{equation}\label{eq:legendre-trafo}
  p^{*} (a) \, := \, \sup_{q\in\RR}
  \bigl( q \ts a - p (q) \bigr).
\end{equation}
Note that, since $p(t)=\infty$ for $t<0$, see
Proposition~\ref{prop:poinwise-conv} below, we are in the particular
situation that
$p^{*} (a) = \sup_{q\geqslant 0} \bigl( q \ts a - p (q) \bigr)$.
We should also mention that $p^{\ast}(\alpha) < \infty $ for
$\alpha < \log\bigl(\frac{3}{2}\bigr)$, again by
Proposition~\ref{prop:poinwise-conv}, such that
$\bigl(-\infty, \log\bigl(\frac{3}{2}\bigr)\bigr)$
is contained in the essential domain of $p^{\ast}$.

For a closed subshift $\XX'\subseteq\XX$ that is invariant under the
left shift, let us define the restricted pressure by
\[
  \cP\left(t\ts \psi \ts | \ts \XX'\right) \, := 
  \lim_{n\to\infty}\myfrac{1}{n}\log\sum_{\omega\in\Sigma^{n}}
  \;\sup_{x\in\langle\omega\rangle\cap\XX'} \! \exp
  \bigl(t\ts \psi^{\pa}_{n}(x)\bigr) .
\]
For any $\omega$ with $\langle\omega\rangle\cap\XX'=\varnothing$, we
set
$\sup_{x\in\langle\omega\rangle\cap\XX'}\exp \left(t\ts
  \psi^{\pa}_{n}(x)\right)=0$.  Clearly,
$\cP\left(t\ts \psi\right)=\cP\left(t\ts \psi \ts | \ts \XX\right)$
gives back the pressure defined above, and
$\cP(t \psi \ts | \ts \XX') \leqslant \cP(t \psi)$ by definition. As
in the previous section, we will be interested in the SFTs $\XX_m$.

Fix a closed invariant subshift $\XX'\subseteq\XX$ and, for every
$n\in\NN$, consider
\[
  a_{n} \, := \, \log\sum_{\omega\in\Sigma^{n}}\,
  \sup_{x\in\langle\omega\rangle\cap\XX'}
  \! \exp \bigl(t\ts \psi^{\pa}_{n}(x) \bigr).
\]
Now, the sequence
$(a_{n})^{\pa}_{n\in\NN}$ is finite and subadditive, which follows by
\begin{align*}
  a^{\pa}_{n+k} \, & = \, \log\sum_{\omega\in\Sigma^{n+k}}\,
            \sup_{x\in\langle\omega\rangle\cap\XX'} \! \exp
            \bigl(t\ts \psi^{\pa}_{n+k}(x)\bigr)\\[1mm]
    &  \leqslant \, \log \biggl(\, \sum_{\omega\in\Sigma^{n}} \,
      \sup_{x\in\langle\omega\rangle\cap\XX'} \!
      \exp \bigl(t\ts \psi^{\pa}_{n}(x)\bigr)
      \sum_{\omega\in\Sigma^{k}} \,
      \sup_{x\in\langle\omega\rangle\cap\XX'} \!
      \exp \bigl( t\ts \psi^{\pa}_{k}(x)\bigr)\biggr)
        \, = \, a^{\pa}_{n} \nts + a^{\pa}_{k} \ts .
\end{align*}
This guarantees that the limit in the definition of
$\cP\left(t\ts \psi \ts | \ts \XX'\right)$ always exists and is given
by the infimum, so we can use
\begin{equation}\label{eq:PressureByInf-1}
  \cP \bigl( t\ts \psi \ts | \ts \XX' \bigr)
  \, := \, \inf_{n\in\NN}\myfrac{1}{n}\log
  \sum_{\omega\in\Sigma^{n}} \,
  \sup_{x\in\langle\omega\rangle\cap\XX'}
  \! \exp \bigl(t\ts \psi^{\pa}_{n}(x)\bigr).
\end{equation}

\begin{prop}\label{prop:Convergence-of-Pressure-1} 
  For\/ $m\in\NN$, consider the function\/
  $p^{\pa}_m \colon \RR \xrightarrow{\quad} \RR$ defined by\/
  $t \mapsto\cP (t\ts \psi \ts | \ts \XX_{m} )$.  Then, for each
  $t \in \RR$, one has
\[
  \cP\left(t\ts \psi\right) \, =
  \lim_{m\to\infty} p^{\pa}_m (t) \ts .
\]
\end{prop}

Since $p^{\pa}_m$ is continuous for all $m \in \mathbb{N}$ by 
\cite{DK}, an immediate consequence of this result is that $p$ is 
a lower semi-continuous function. Thus, the same holds for its 
Legendre transform $p^{\ast}$.

\begin{proof}[Proof of
    Proposition~\textnormal{\ref{prop:Convergence-of-Pressure-1}}]
  Since $\cP\left(t\ts \psi|\XX_{m}\right)$ is monotonically
  increasing in $m$, the limiting expression in
  Proposition~\ref{prop:Convergence-of-Pressure-1} exists.  From the
  fact that
  $ \cP (t\ts \psi \ts | \ts \XX_{m}) \, \leqslant \, \cP (t \ts \psi
  )$ for every $m\in\NN$, we obtain
\[
   \cP (t\ts \psi )
   \, \geqslant \, \limsup_{m\to\infty}
   \cP (t\ts \psi \ts | \ts \XX_{m} ) \ts . 
\]
Next, we prove
$ \cP (t\ts \psi ) \leqslant \liminf_{m\to\infty}\cP (t\ts \psi\ts |
\ts \XX_m )$ in several steps.  Our approach is to find, for each word
$\omega\in\Sigma^n\setminus \Sigma_m^n$, a corresponding word
$\omega'\in\Sigma_{m}^n$, and compare
$\sup_{x\in\langle\omega\rangle}t\ts \psi^{\pa}_{n}(x)$ with
$\sup_{x\in\langle\omega'\rangle\cap\XX_m}t\ts \psi^{\pa}_{n}(x)$.
This will be done in Lemma~\ref{lem:comp-after-algorithm}.
Furthermore, for a given $\omega'\in\Sigma_{m}^n$, we will estimate
the number of words $\omega\in\Sigma^n\setminus \Sigma_m^n$ which will
be compared with $\omega'$.  This will be done in Lemma
\ref{lem:number}.
 
To construct such an $\omega'$ for
$\omega=\omega^{\pa}_1 \! \cdots \omega^{\pa}_n$, we use the following
algorithm.  Start at $\omega^{\pa}_1$.  Look at the first letter where
$0$ or $1$ has appeared $(m\! +\! 1)$ times in a row.  Say this
happens at $\omega^{\pa}_\gamma$. Then, build the new word
$\widetilde{\omega}:=\omega^{\pa}_1 \! \cdots\omega^{\pa}_{\gamma-1}
\widehat{\omega}^{\pa}_\gamma\!\cdots\widehat{\omega}^{\pa}_n$.
Repeat the algorithm with $\widetilde{\omega}$ and keep repeating
until the final word $\omega'$ lies in $\Sigma_{m}^n$.  We denote the
map given by this algorithm by $h$, so $h\left(\omega\right)=\omega'$.

\begin{lemma}\label{lem:number}
  For\/ $\omega'\in\Sigma_{m}^n$, one has\/
  $\,\card\left\{\omega\in\Sigma^n\setminus \Sigma_m^n :
    h\left(\omega\right) = \omega'\right\} < 2^{\left\lfloor
      n/m\right\rfloor}$.
\end{lemma}

\begin{proof}
  Each $\omega'\in \Sigma_m^n$ contains at most
  $\left\lfloor n/m\right\rfloor$ single{\ts}-letter subwords of
  length $m$.  The following algorithm gives a possibility to find
  pre-images of $\omega'$.  Let $\gamma(1),\ldots, \gamma(i)$ be the
  integers such that
  $\omega'_{\gamma(j)-m}=\ldots=\omega'_{\gamma(j)-1}$ for some
  $1\leqslant j \leqslant i$, with $i$ the total number of such
  sequences.  We denote by $v^{\pa}_{\nts j,1}$ the identity and set
  $v^{\pa}_{\nts j,2}\left(\omega'\right):= \omega_1'\cdots
  \omega_{\gamma(j)-1}'\widehat{\omega}_{\gamma(j)}'
  \cdots\widehat{\omega}_{n}'$.  Then, we have that
  $v^{\pa}_{1,k_1}\circ\cdots\circ v^{\pa}_{i,k_i}\left(\omega'\right)$ with
  $k_{\ell}\in\left\{1,2\right\}$ are all pre-images of $h$. Since
  there are $2^i$ possibilities to choose
  $\left\{k_1,\ldots, k_i\right\}$, there are 
  $2^i$ pre-images (with
  all but one in $\Sigma^n\setminus\Sigma_m^n$).
\end{proof}

\begin{lemma}\label{lem:comp-after-algorithm}
  For any pair\/ $\omega\in\Sigma^n \setminus \Sigma_m^n$ and\/
  $\omega'=h(\omega)\in\Sigma_{m}^n$, we have
\[
  \sup_{x\in\langle \omega'\rangle\cap\XX_m} \! \psi^{\pa}_{n}(x)+ 4
  \left\lfloor n/m\right\rfloor + 2^{m+2} \: \geqslant \,
  \sup_{y\in\langle\omega\rangle} \psi^{\pa}_{n} (y) \ts .
\]
\end{lemma}

\begin{proof}
Clearly, we have  
\[
  \sup_{y\in\langle\omega\rangle}\psi^{\pa}_{n}(y)\; - \!
  \sup_{x\in\langle\omega'\rangle\cap\XX_m}\! \psi^{\pa}_{n}(x)
  \: \leqslant \,  \sum_{\ell=0}^{n-1}
  \sup_{\substack{y\in\langle2^{\ell}\omega\rangle \\
    z\in\langle2^{\ell}\omega'\rangle\cap\XX_m}} \!\!
  \! \bigl(\psi(y)-\psi (z) \bigr),
\]
where for $\omega = \omega^{\pa}_1 \cdots \omega^{\pa}_n$ we have
 used the notation $2^{\ell}\omega = \omega^{\pa}_{1+\ell} \cdots
 \omega^{\pa}_n$, for $0\leqslant \ell \leqslant n-1$.

In the next steps, we aim to estimate
\[
  \sup_{\substack{y\in\langle2^{\ell}\omega\rangle\\
      z\in\langle2^{\ell}\omega'\rangle\cap\XX_m}}\!\!
   \bigl(\psi(y)-\psi (z) \bigr),
\]
separately for each $0\leqslant \ell \leqslant n-1$.  Let $ \gamma(j)$
be the position of the first digit which gets inverted $j$ times by
the above algorithm.

With $i$ defined as in the proof of Lemma~\ref{lem:number}, 
it is $i\leqslant \left\lfloor n/m\right\rfloor$. For
$\ell\in\left\{0,\ldots,\gamma(1)-2-m\right\}$, we have 
\begin{align*}
  \sup_{\substack{y\in\langle2^{\ell}\omega\rangle\\
  z\in\langle2^{\ell}\omega'\rangle\cap\XX_m}}
  \! \bigl(\psi(y)-\psi (z)\bigr) \:
  & \leqslant \sup_{y,z\in\langle\omega^{\pa}_{\ell+1}\cdots\,
    \omega^{\pa}_{\gamma(1)-1}\rangle}\bigl(\psi(y)-\psi(z)\bigr)\\[1mm]
  & \, =  \sup_{y\in\langle\omega^{\pa}_{\ell+1}\cdots\,
          \omega^{\pa}_{\gamma(1)-1}\rangle} \! \psi(y) \; -
    \inf_{z\in\langle\omega^{\pa}_{\ell+1}\cdots\,\omega^{\pa}_{\gamma(1)-1}\rangle}
    \! \psi(z) \ts .
\end{align*}
Here, we have
$\langle\omega^{\pa}_{\ell+1} \! \cdots \omega^{\pa}_{\gamma(1)-1}
\rangle\subset\Sigma_m^{\gamma(1)-\ell-1}$ and, by the choice of
$\ell$, we also have the inequality $\gamma(1)-\ell-1> m$.  This
implies for all
$x\in \langle\omega^{\pa}_{\ell+1} \! \cdots
\omega^{\pa}_{\gamma(1)-1}\rangle$ with
$\ell\in\left\{0,\ldots,\gamma(1)-2-m\right\}$ that
$x\in\left[2^{-m-1},1-2^{-m-1}\right]=: J_m$, which follows directly
by considering the dual representation of the unit interval.

For the next estimate, we use the elementary formula
\[
  \left|\psi\left(x+h\right)-\psi(x)\right|
  \, \leqslant \, h \max_{u\in\left[x,x+h\right]}
  \left|\psi^{\ts\prime}\left(u\right)\right|  .
\] 
For $y,z$ as above (i.e., the points where the supremum and the
infimum are attained), we have
$\left|y-z\right|\leqslant 2^{-\gamma(1)+\ell+1}$.  Since
$y,z\in J_m$, we see via Lemma~\ref{lem:deriv} that
$\left|\psi^{\ts\prime}(y)\right|\leqslant 2^{m+2}$.  Combining these
observations gives
\[
  \sup_{\substack{y\in\langle2^{\ell}\omega\rangle\\
      z\in\langle2^{\ell}\omega'\rangle\cap\XX_m}}
  \! \bigl(\psi(y)-\psi (z) \bigr)
  \, \leqslant \, 2^{m+\ell+3-\gamma(1)}.
\]

Recall that $\psi(x)=\psi (\widehat{x}\ts )$ holds for all $x\in\XX$.
In analogy to above, for any choice of
$\ell\in\left\{\gamma(j)-1,\ldots,\gamma(j+1)-2-m\right\}$ with
$j\in\left\{1,\ldots, i-1\right\}$, we then have
\[
    \sup_{\substack{y\in\langle2^{\ell}\omega\rangle\\
           z\in\langle2^{\ell}\omega'\rangle\cap\XX_m}}
  \bigl( \psi(y)-\psi (z) \bigr)
   \, \leqslant 
   \sup_{y,z\in\langle\omega^{\pa}_{\ell+1}\cdots
         \,\omega^{\pa}_{\gamma(j+1)-1}\rangle}
    \! \bigl(\psi(y)-\psi (z)\bigr)
 \, \leqslant \, 2^{m+\ell+3-\gamma(j+1)}.
\]
When $\ell\in\left\{\gamma(j)-1-m,\ldots, \gamma(j)-2\right\}$ with
$j\in\left\{1,\ldots,i\right\}$, the word $2^{\ell}\omega$ starts with
a consecutive sequence of a single letter that is one digit longer
than the corresponding sequence in $2^{\ell}\omega'$.  Note that the
consecutive letter in $2^{\ell}\omega$ and in $2^{\ell}\omega'$ does
not have to be the same but this issue can be handled by the fact that
$\psi(x)=\psi (\widehat{x}\ts )$.  By the monotonicity on
$\bigl[ 0, \frac{1}{2}\bigr]$ and on $\bigl[ \frac{1}{2} , 1 \bigr]$,
we obtain
\[
  \sup_{\substack{y\in\langle2^{\ell}\omega\rangle\\
      z\in\langle2^{\ell}\omega'\rangle\cap\XX_m}}\!\!
  \bigl(\psi(y)-\psi (z)\bigr)  \, \leqslant \, 0 \ts .
\]

Finally, we consider $\ell\in\left\{\gamma(i)-1,\ldots,n-1\right\}$
with $\gamma (i)$ as above. This gives
\[
  \sup_{\substack{y\in\langle2^{\ell}\omega\rangle\\
      z\in\langle2^{\ell}\omega'\rangle\cap\XX_m}}
    \!\! \bigl(\psi(y)-\psi(z)\bigr)
    \, =  \sup_{y\in\langle\omega^{\pa}_{\ell+1}
      \cdots\,\omega^{\pa}_n\rangle} \! \psi(y) \; - \!
    \inf_{z\in\langle\omega^{\pa}_{\ell+1}
       \cdots\,\omega^{\pa}_n\rangle\cap\XX_m}
    \! \psi(z) \ts .
\]
We note that each of the words
$\omega^{\pa}_{\ell+1} \! \cdots\omega^{\pa}_n$ with
$\ell\in\left\{\gamma(i)-1,\ldots,n-1\right\}$ starts with at most
$m$ times the same digit.  If $\ell<n-m$, the word
$\omega^{\pa}_{\ell+1}\! \cdots \omega^{\pa}_n$ is longer than $m$
digits, which implies that the supremum of $\psi$ on the cylinder
$\langle\omega^{\pa}_{\ell+1}\!\cdots\omega^{\pa}_n\rangle$ is
attained in $J_m$.

If $\ell\geqslant n-m$, the supremum of $\psi$ on the cylinder
$\langle\omega^{\pa}_{\ell+1}\!\cdots\omega^{\pa}_n\rangle$ is
attained at the word starting with
$\omega^{\pa}_{\ell+1}\!\cdots\omega^{\pa}_n$ followed by an infinite
sequence of the letter $\widehat{\omega}^{\pa}_{\ell+1}$.  This
follows from the fact that $\psi$ is monotonically increasing on the
interval $\bigl[ 0, \frac{1}{2} \bigr]$ and monotonically decreasing
on $\bigl[ \frac{1}{2} ,1 \bigr]$.  This implies in particular that
the supremum is also attained on $J_m$ in this case.  By the
restriction $z\in\XX_m$ for the infimum, it follows directly that
$z\in J_m$.  Thus, for $\ell\in \{\gamma(i)-1,\ldots,n-1 \}$, an
analogous argument as above gives
\[
  \sup_{\substack{y\in\langle2^{\ell}\omega\rangle\\
      z\in\langle2^{\ell}\omega'\rangle\cap\XX_m}}
  \!\! \bigl( \psi (y)-\psi (z) \bigr) 
   \, \leqslant \, 2^{m+\ell +2 -n}.
\]

As usual, for $\gamma'<\gamma$, we define the sum
$\sum_{\ell=\gamma}^{\gamma'}$ over any quantity to be zero.  Putting
things together, with the integer $i$ from above, we get
\begin{align*}
  \sum_{\ell=0}^{n-1} &
  \sup_{\substack{y\in\langle2^{\ell}\omega\rangle\\
  z\in\langle2^{\ell}\omega'\rangle\cap\XX_m}}
  \!\! \bigl( \psi(y)-\psi(z)\bigr) \\[2mm]
  \leqslant \;\,& \! \sum_{\ell=0}^{\gamma(1)-m-2} 2^{m+\ell+3-\gamma(1)}
      \, + \, \sum_{j=1}^{i-1} \sum_{\ell=\gamma(j)-1}^{\gamma(j+1)-2-m}
       2^{m+\ell+1-\gamma(j+1)}
  \;\,  +\sum_{\ell=\gamma(i)-1}^{n-1} 2^{m+\ell+2-n} \\[2mm]
  \leqslant \;\, & \sum_{\ell'=0}^{\infty}2^{-\ell'+1}+
      \biggl(\sum_{j=1}^{i-1} \sum_{\ell'=0}^{\infty}2^{-\ell'+1}\biggr)
      +\sum_{\ell'=0}^{\infty} 2^{-\ell'+m+1} 
      \, \leqslant \;\,  4 i +2^{m+2} \, \leqslant \,
      4 \left\lfloor n/m\right\rfloor+2^{m+2} .
\end{align*}
For the last estimate, we used the fact that
$i\leqslant \left\lfloor n/m\right\rfloor$, as in the proof of
Lemma~\ref{lem:number}.
\end{proof}

If we apply Lemmas~\ref{lem:number} and \ref{lem:comp-after-algorithm}
for $t\geqslant 0$, we obtain
\[
   \sum_{\omega\in\Sigma^n}\, \sup_{x\in\langle\omega\rangle}
    \exp (t\ts \psi^{\pa}_{n}(x) )  \, \leqslant 
    \sum_{\omega'\in\Sigma_m^n} \! \! 2^{\left\lfloor n/m\right\rfloor}
    \exp \bigl( t (4 \left\lfloor n/m\right\rfloor+2^{m+2} ) \bigr) \!
    \sup_{x\in\langle\omega'\rangle\cap\XX_m}
    \! \! \exp (t\ts \psi^{\pa}_{n}(x)) \ts .
\]
This implies 
\begin{align*}
    \cP ( t\ts \psi) \, & \leqslant \lim_{n\to\infty}\myfrac{1}{n} \log 
     \Bigl( 2^{\left\lfloor n/m\right\rfloor} \exp \bigl( t ( 2^{m+2} + 
     4 \lfloor n/m \rfloor )\bigr) \! \sum_{\omega'\in\Sigma_m^n} 
       \, \sup_{x\in\langle\omega'\rangle\cap\XX_m}
       \! \exp (t\ts \psi^{\pa}_{n}(x) ) \Bigr)\\[1mm]
    & =  \lim_{n\to\infty} \frac{\left\lfloor n/m\right\rfloor\log(2)
     + t (2^{m+2} + 4 \lfloor n/m \rfloor )}{n} 
     +  \cP (t\ts \psi \ts\lvert\ts \XX_m )\\[2mm]
    & = \, \frac{\log(2) + 4 \ts t}{m} 
            + \cP (t\ts \psi\ts\lvert\ts\XX_m ) \ts .
\end{align*}
Consequently, $ \cP (t\ts \psi ) \leqslant \liminf_{m\to\infty}\cP 
(t\ts \psi\ts\lvert\ts\XX_m )$ holds for $t\geqslant 0$.

Next, we consider the case $t< 0$, where we have to show that
$\lim_{m\to\infty}\mathcal{P}\left(t\psi\lvert\mathbb{X}_m\right)=\infty$.
We let $\omega_{n,m}\coloneqq \omega_{m}^{[n]}$ denote the prefix of
$\omega_{m}\coloneqq \overline{0^m1^{m} }$ of length
$n\in \mathbb{N}$.  For $t<0$, we then have
\begin{align*}
  \mathcal{P}\left(t\psi\lvert\mathbb{X}_m\right)
  \, \geqslant &
     \lim_{n\to\infty} \myfrac{1}{n}\log
     \sup_{x\in\left<\omega_{n,m}\right>\cap\mathbb{X}_m}
     \exp \bigl( t\psi_{n} ( x )\bigr)
     \, = \, t\lim_{n\to\infty}\myfrac{1}{n}
     \inf_{x\in\left<\omega_{n,m}\right>\cap\mathbb{X}_m}
     \psi_{n}\left(x\right)\\
  \, \geqslant \,&
      t\lim_{n\to\infty}\myfrac{1}{n}
      \Biggl( \left\lfloor \myfrac{n}{m}\right
      \rfloor  \sup_{x\in\left<0^m\right>\cap\mathbb{X}_m}
      \psi_{m} (x) +m\biggr)
        \, = \, \myfrac{t}{m} \sum_{k=1}^{m}
        \psi \bigl( 2^{-k} \bigr).
\end{align*}
Using the asymptotic behaviour
$ \lim_{x\searrow 0}\frac{\psi(x)}{\log\left(x\right)}= 2 , $ a
straightforward calculation then shows
$\lim_{m \to \infty} \frac{1}{m} \sum_{k=1}^m \psi (2^{-k}) = -
\infty$, which gives the desired result.  
\end{proof}

\section{Proof of  the first part of
  Theorem~\ref{thm:MAIN-MF-1}}\label{Sec:Proof-of-Thm}

If we restrict our attention to any of the SFTs $\XX_m$, there is a
natural correspondence between the local dimension of $\nu$ and the
Birkhoff average of $\psi$ due to
Corollary~\ref{Coro:SFT-local-dim-birkhoff-relation}. Moreover,
standard multifractal formalism for H{\"o}lder continuous potentials,
compare~\cite{MR2719683}, provides us with the relation
\begin{equation}\label{eq:Haus-for-Xm}
   \dim_{ \mathrm{H},\tau} \Bigl\{ x\in\XX_m :
    \lim_{n\to\infty}\frac{\psi^{\pa}_{n}(x)}{n}
    =  \alpha\Bigr\} \, = \, \max
   \Bigl\{\frac{-p^{*}_{m} (\alpha)}{\log(2)}, 0\Bigr\},
\end{equation}
which holds for both $\tau\in \{\mathrm{E},2\}$ because the metric
spaces with the two different metrics restricted to $\XX_m$ are
bi-Lipschitz equivalent.  To show Theorem~\ref{thm:MAIN-MF-1}, we need
to make sure that these relations extend properly to the full shift
$\XX$. Intuitively, $\XX_m$ captures the relevant information about
$\nu$ (in the limit of large $m$) because it is designed in such a way
that, given $\omega \in \Sigma^n$, $\nu(\langle \omega \rangle)$ is
small if $\langle \omega \rangle \cap \XX_m = \varnothing$. In parts,
this is made more precise by
Proposition~\ref{prop:Convergence-of-Pressure-1} showing that the
restricted pressure function indeed converges to the full pressure
function. Some additional work is necessary to ensure that the
information which we obtain from $\XX_m$ with $m \in \NN$ is
consistent with the Hausdorff dimension of the full level sets.

In what follows, we will prove the two identities from
Theorem~\ref{thm:MAIN-MF-1}. We postpone the proof of all the
properties of the spectrum to Section~\ref{Sec:Proof-of-Thm-II}, as we
shall need further properties of the pressure function $p$. The latter
will be derived in Section~\ref{Sec:p-properties}.

\begin{proof}[Proof of  the first part of 
Theorem~\textnormal{\ref{thm:MAIN-MF-1}}]
Let us begin with proving the upper bounds in
\eqref{eq:Birkhoff-spectrum} and \eqref{eq:dimension-spectrum} by
giving an upper bound for the Hausdorff dimension of a sufficiently
large superset.

For arbitrary $n \in \NN$, we introduce a neighbour relation on the
set of words $\Sigma^n$ as follows. First, $\omega$ is called a
neighbour of $\omega'$ if the intersection of $\langle \omega \rangle$
and $\langle \omega' \rangle$ as subsets of $[0,1]$ consists of
precisely one point, that is, if the corresponding intervals are
adjacent in $[0,1]$. Then, for each $n \geqslant 2$ and $x\in \XX$, we
can choose a neighbour $\omega_x^{n-1}$ of $C_{n-1}(x)$ such that
$ B^{\pa}_{\mathrm{E}}(x,2^{-n}) \subset C_{n-1}(x) \cup \langle
\omega_x^{n-1} \rangle$.  Next, let
\[
   \widetilde{\psi}^{\pa}_{n} (x) \, := \, \log \Bigl( \,
    \sup_{ z \in C_{n-1}(x)} \! \exp \bigl(\psi^{\pa}_{n-1}
    (z) \bigr) \, + \! \sup_{ z'\in \langle\omega_x^{n-1}\rangle}
    \! \exp \bigl( \psi^{\pa}_{n-1} (z' \ts ) \bigr) \Bigr) ,
\]
and,  for $n\in \NN$ and $\alpha \in\RR$,  consider the set
\[
    \cG\left(n,\alpha\right) \,  := \, \bigl\{  x\in \TT : 
    \widetilde{\psi}^{\pa}_{n} (x) - n  \alpha > 0 \bigr\} .
\]

For $\alpha \in \RR$ and $\varepsilon>0$, we define
$\varGamma_n :=\left\{\omega\in\Sigma^n : \cG (n,\alpha-\varepsilon
  )\cap\langle\omega\rangle \neq\varnothing\right\}$. Then, for fixed
$q\geqslant 0$ and any
$s > \frac{q (\alpha-\varepsilon)-\cP(q\psi)}{-\log(2)}$, we obtain
\begin{align*}
            \sum _{n>m}\,
            \sum_{\omega\in\varGamma^{\pa}_{n}}
            \left|\langle\omega\rangle\right|^{s} \, & = 
            \sum _{n>m} \, \sum_{\omega\in\varGamma^{\pa}_n}
            \ee^{-sn\log(2)}
            \, \leqslant \sum _{n>m} \, 
            \sum_{\omega\in\varGamma^{\pa}_n} \ee^{-sn\log(2)}
            \! \sup_{x\in\langle\omega\rangle}\exp 
            \bigl(q\widetilde{\psi}^{\pa}_{n}(x) 
                - n q (\alpha-\varepsilon)\bigr)\\[1mm]
         & \leqslant \sum _{n>m} \, \sum_{\omega\in\Sigma^{n}}
            \sup_{x\in\langle\omega\rangle} \exp
            \bigl(-sn\log(2)+q\widetilde{\psi}^{\pa}_{n} (x)
               -n q (\alpha-\varepsilon)\bigr)\\[1mm]
         & \leqslant  \sum _{n>m} \exp \biggl( -n
           \Bigl( s\log(2) + q (\alpha-\varepsilon) - \myfrac{1}{n}
           \log  \! \sum_{\omega\in\Sigma^{n-1}} \! 3 \!
             \sup_{x\in\langle\omega\rangle}
             \ee^{q \psi^{\pa}_{n-1}(x)} \Bigr)\biggr)\\[2mm]
         & \quad \xrightarrow{\; m\to\infty \;} 0 \ts ,
\end{align*}
where we have used the definition of $-\cP\left(q\psi\right)$ and the
fact that any $\omega \in \Sigma^{n-1}$ can be a neighbour of at most
two distinct words in $\varGamma_{n-1}$.  Consequently, the
$s$-dimensional Hausdorff measure (with respect to both
$\varrho_{\mathrm{E}}$ and $\varrho_{2}$) equals zero.  Since this
holds for an arbitrary
$s > \frac{ q (\alpha-\varepsilon)- \cP (q\psi )}{-\log(2)}$ with
$q\geqslant 0$, we may conclude for the Hausdorff dimension of the
$\limsup$-set that
\[
     \dim^{\pa}_{ \mathrm{H},\tau}\limsup_{n\to \infty}
     \cG(n,\alpha - \varepsilon ) \, \leqslant \, \inf_{q\geqslant 0}
     \frac{q (\alpha-\varepsilon ) - \cP (q\psi )}{-\log(2)}
     \, = \, \frac{p^{*}(\alpha-\varepsilon)}{-\log(2)}.
\] 

Let us proceed to the final step. Making use of the immediate
inequality $\widetilde{\psi}^{\pa}_n\geqslant \psi^{\pa}_{n-1}$, and
combining the inclusions
$B^{\pa}_2(x,2^{-n}) \subseteq B^{\pa}_{\mathrm{E}}(x,2^{-n})
\subseteq C_{n-1}(x) \cup \langle \omega_{x}^{n-1}\rangle$ with the
inequality
$ \nu(C_n(x))\leqslant 2^{-n} \sup_{y \in C_n(x)}
\exp(\psi^{\pa}_{n}(y))$ from Lemma~\ref{Lemma:bounds-nu}, we can
deduce that, for each $\varepsilon > 0$, each of the sets
\[
    \Bigl\{ x \in  \TT : \limsup_{ n\to\infty}\frac{\psi^{\pa}_{n}(x)}{n}  
    \geqslant \alpha  \Bigr\} \quad \text{and} \quad
    \Bigl \{ x \in  \TT : \liminf_{ n\to\infty}\frac{\log
    \bigl(\nu(B_{\tau}(x,2^{-n}))\bigr)}{\log(2^{-n})}  
    \leqslant 1-\frac{\alpha}{\log(2)} \Bigr\} , 
\]
for $\tau\in \{2,\mathrm{E}\}$, is a subset of
$ \limsup_{n\to \infty}\cG(n,\alpha - \varepsilon)$. Hence, by the
lower semi-continuity of the Legendre transform, we obtain via
$\varepsilon \,\mbox{\footnotesize $\searrow$}\, 0$ the desired upper
bounds in \eqref{eq:Birkhoff-spectrum} and
\eqref{eq:dimension-spectrum}.

For the lower bound, fix $m\in\NN$.  Then, for
$\alpha  < \log\bigl(\frac{3}{2}\bigr)$ and 
$\tau \in \{\mathrm{E},2\}$, we have
\begin{align*}
    \birk_{ \tau} \left(\alpha\right) \, & \geqslant \, 
    \dim^{\pa}_{ \mathrm{H},\tau} \Bigl\{ x\in\XX_m : 
    \lim_{n\to\infty}\frac{\psi^{\pa}_{n}(x)}{n}
      =  \alpha \Bigr\} \\[1mm]
   & = \, \max \Bigl\{ \frac{-p^{*}_{m} (\alpha )}{\log(2)},0 \Bigr\}
     \; \xrightarrow{\; m\to\infty \;} \; \max
     \Bigl\{ \frac{-p^{*} (\alpha )}{\log(2)}, 0 \Bigr\} ,
\end{align*}
where the equality is due to \eqref{eq:Haus-for-Xm} and the
convergence part follows from
$\lim_{m \rightarrow \infty} p^{\pa}_{m}(t) = p(t)$, see
Proposition~\ref{prop:Convergence-of-Pressure-1}, in conjunction with
\cite[Prop.~4.1]{MR2719683}.  In Section~\ref{Sec:Proof-of-Thm-II}, we
shall see that $p^{\ast}\bigl(\log\bigl(\frac{3}{2}\bigr)\bigr) = 0$
and $p^{\ast}(\alpha) = \infty$ for
$\alpha > \log \bigl(\frac{3}{2}\bigr)$, so the lower bound is trivial
for $\alpha \geqslant \log \bigl(\frac{3}{2}\bigr)$.

For the remaining case, we use
Corollary~\ref{Coro:SFT-local-dim-birkhoff-relation} to deduce that,
for any $\tau \in\{\mathrm{E},2\}$ and $m \in \NN$, one has
\[
    \cF_{\tau} (\alpha) \, \supseteq \, \big\{ x \in \XX_m : 
    \locdim_{\nu,\tau}(x) = \alpha \big\} \, = \, 
    \Bigl\{ x \in \XX_m : \lim_{n \rightarrow \infty} 
      \frac{\psi^{\pa}_{n}(x)}{n} = (1 - \alpha)
    \log(2) \Bigr\} ,
\]
where $\cF_{\tau} (\alpha) = \big\{ x \in \XX : \locdim_{\nu,\tau}(x)
 = \alpha\big\}$.
This yields the lower bound in \eqref{eq:dimension-spectrum}. 
 \end{proof}

\section{Further properties of the pressure
   function}\label{Sec:p-properties}

Before we can continue with the proof of
Theorem~\ref{thm:MAIN-MF-1}, we need to establish some further
properties of the pressure function $p$, which might also be of
independent interest.

\begin{prop}\label{prop:poinwise-conv}
  The function\/ $p \colon \RR \xrightarrow{\quad} \RR\cup\{+\infty\}$
  with\/ $p(t) = \cP\left(t\ts \psi\right)$ is well defined and
  convex.  The sequence of real-analytic convex functions\/
  $p^{\pa}_{m}$ from
  Proposition~\textnormal{\ref{prop:Convergence-of-Pressure-1}}
  converges pointwise to\/ $p$; see
  Figure~\textnormal{\ref{fig:The-graph-of-pressureFunction-1}} for an
  illustration of the graph of\/ $p$.

  Further, the domain\/ $\{t\in \RR\colon p(t)<\infty\}$ of\/ $p$ is
  equal to\/ $[0,+\infty)$, and one has
\begin{align}
   p(0) \, & = \,\log(2) \ts ,\label{eq:prop-2.5aa}\\[1mm]
   p(1) \, & = \,\log(2) \ts ,\label{eq:prop-2.5aaa}\\[1mm]
    \lim_{t\to\infty} p(t) - \log(3/2) \ts t 
     & = \, 0 \ts , \label{eq:prop-2.5a}\\
    p^{\ts \prime} ( 0+ ) \, & = \, -\log(2) \ts .\label{eq:prop-2.5c}
\end{align} 
\end{prop}

\begin{figure}[ht]
  \includegraphics[width= 0.7\textwidth]{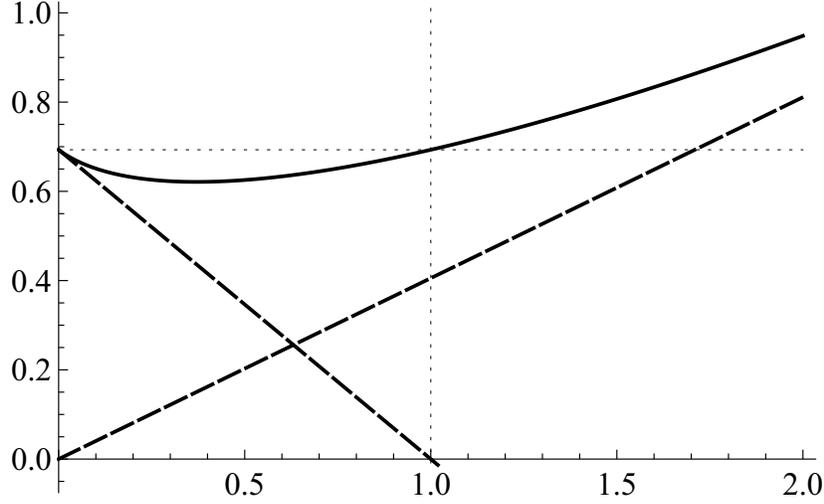}
  \caption{The graph of the pressure function $p$ (solid line) with
    the two asymptotes $x\protect\mapsto \log(3/2) \ts x$ 
    and $x\protect\mapsto\left(1-x\right)\log(2)$ (dashed lines). The
    dotted lines are added to illustrate $p(0) = p(1) =
    \log(2)$.\label{fig:The-graph-of-pressureFunction-1}}
\end{figure}

\begin{proof}
  The pointwise convergence follows from
  Proposition~\ref{prop:Convergence-of-Pressure-1}, while convexity
  and analyticity of the functions $p^{\pa}_{m}$ follow from the
  general thermodynamic formalism for H\"older continuous functions;
  compare \cite{DK}. Convexity of $p$ follows from the fact that the
  pointwise limit of convex functions is again convex.
  
  Next, $p (t)=\infty$ for $t<0$ follows from
\[
    \myfrac{1}{n} \log \sum_{\omega\in\Sigma^{n}}
    \sup_{x\in\langle\omega\rangle} \exp (t\ts \psi^{\pa}_{n}(x))
    \, \geqslant \, \myfrac{1}{n} \log 
    \! \sup_{x\in\langle0^n \rangle} \!
    \exp (t\ts \psi^{\pa}_{n}(x)) \, = \, \infty \ts .
\]
The finiteness of $p$ for $t\in [0,+\infty)$
follows from the obvious estimate $p(t)\leqslant \log(2) \ts (1+t)$
with equality for $t=0$, the latter giving \eqref{eq:prop-2.5aa}.

To establish \eqref{eq:prop-2.5aaa}, we follow the approach of
\cite[p.~19]{Bow}, where we first prove $p(1)\geqslant \log(2)$.  For
$x\in\XX$, Lemma~\ref{Lemma:bounds-nu} implies
\[
     \frac{\nu(C_n(x))}{\exp \bigl(-n\log(2)+
     \sup_{y\in C_n (x)}\psi^{\pa}_{n}(y) \bigr)} 
     \, \leqslant \, 1 \ts .
\]
Summing the measure over all possible cylinder sets of length $n$
gives $1$ and thus
 \[
      \exp \Bigl(-n\log(2)+\log \sum_{\omega\in\Sigma^n}
      \sup_{x\in\langle\omega\rangle}\!
      \exp(\psi^{\pa}_{n}(x)) \Bigr) \, \geqslant \, 1 \ts ,
 \]
  and we get
\[
     \myfrac{1}{n}\log  \! \sum_{\omega\in\Sigma^n}
     \sup_{x\in\langle\omega\rangle} \!
     \exp (\psi^{\pa}_{n}(x)) \, \geqslant \, \log(2) \ts .
 \]
 Letting $n \to \infty$ then gives the first inequality.

 Next, we prove $p(1)\leqslant\log(2)$. Let $x\in\XX_m$ and
 $y\in C_n(x)\cap \XX_m$. Then, a combination of 
 Lemmas~\ref{Lemma:SFT-sup-inf-bounds} and \ref{Lemma:bounds-nu} 
 shows that there exists an $R>0$ such that
 \[
     \frac{\nu(C_n(x))}{\exp \bigl(-n\log(2)+
     \sup_{y\in C_n(x)\cap\XX_m}\psi^{\pa}_{n}(y)\bigr)} 
     \, \geqslant \, R \ts .
\]
Summing the measure of all possible cylinder sets in $\XX_m$ of length
$n$ gives a number $\leqslant 1$. Consequently,
\[
     R  \exp \Bigl(-n\log(2)+\log \sum_{\omega\in\Sigma^n}
     \sup_{x\in\langle\omega\rangle\cap\XX_m}\!
     \exp(\psi^{\pa}_{n}(x)) \Bigr) \, \leqslant \, 1 \ts ,
\]
which implies 
\[
    \myfrac{1}{n} \log \sum_{\omega\in\Sigma^n}
    \, \sup_{x\in\langle\omega\rangle\cap\XX_m}
    \! \exp(\psi^{\pa}_{n}(x))
    \, \leqslant \, \log(2)-\frac{ \log (R)}{n}.
\]
Letting $n \to \infty$ gives
$\cP (\psi\ts\lvert\ts \XX_m )\leqslant \log(2)$.
Proposition~\ref{prop:Convergence-of-Pressure-1} now gives the second
inequality.

Let us next establish \eqref{eq:prop-2.5a}.  Observe first that, by an
application of Proposition~\ref{maxBeta}, we have
\[
    p (t)  \, \geqslant \lim_{n\to\infty}\myfrac{1}{n}\log 
    \exp \bigl( t \psi^{\pa}_{n} ( \overline{01} \, ) \bigr)
     \, = \, \log(3/2) \, t
\]
for all $t>0$. Hence, it remains to show 
\[
     \lim_{t\to\infty}p(t) - \log(3/2) \, t \, \leqslant \, 0 \ts .
\]
We will prove this inequality in a series of lemmas, following some
ideas developed in \cite{KS}.

\begin{lemma}\label{lem:equiv-to-prev-est}
  Let\/ $y$ be as in Eq.~\eqref{eq:def-y-hat-y}.  Then, there exists an\/
  $\varepsilon>0$ such that
 \[
     \bigl( \psi+\psi^{(2)} \bigr) (y)
     \,\geqslant\, \bigl( \psi+\psi^{(2)}\bigr)
     (011z ) + \varepsilon 
 \]
 holds for all\/ $z\in\XX$.  Furthermore, for all\/ $k\in\NN$ and\/
 $z\in\XX$, we have
 \[
      \bigl(\psi+\psi^{(2)} \bigr) (y)
     \,\geqslant\, \bigl( \psi+\psi^{(2)} \bigr)
     (y^{[2k]}1z ) \ts .
 \]
 \end{lemma}
 
\begin{proof}
  The proof can be done similarly to that of Lemma~\ref{lem:prev-est}.
  An argument analogous to the one leading to
  \eqref{eq:prev-est-1} yields $\psi^{(2)}(y)=\psi(y)$ and
  $\psi^{(2)} \bigl( y^{[2k]}1z \bigr)= \psi \bigl( y^{[2k-1]}0
  \widehat{z} \ts \bigr)$.  Consequently, the two estimates claimed
  are equivalent to
\[
     \psi\bigl(y^{[2k-1]}0\widehat{z}\ts \bigr)
     +\psi\bigl(y^{[2k]}1z\bigr) \, \leqslant\, 2 \psi(y) - W,
\]
where $W=\varepsilon$ if $k=1$ and $W=0$ otherwise. 

Observe that we have  
\begin{align*}
     \psi\bigl(y^{[2k-1]}0\widehat{z}\ts \bigr) 
        +\psi\bigl(y^{[2k]}1z\bigr) \, & =\, 
     \psi\bigl(y^{[2k-1]}0\widehat{z}\ts\bigr)
        +\psi\bigl(y^{[2k-1]}11z\bigr) \\[2mm]
     & \leqslant\, \psi\bigl(y^{[2k-1]}0\overline{1}\ts\bigr)
        +\psi\bigl(y^{[2k-1]}11\overline{0}\ts \bigr) 
     \, =\, \psi\bigl(y^{[2k]}\overline{0}\ts \bigr)
        +\psi\bigl(y^{[2k]}1\overline{0}\ts \bigr)\\[1mm]
     & = \, 2 \Bigl( \myfrac{1}{3} \psi
              \bigl(y^{[2k]}\overline{0}\ts\bigr)
        + \myfrac{2}{3}\psi\bigl(y^{[2k]}1\overline{0}\ts \bigr) \Bigr) 
        - W^{\pa}_{k}  \, \leqslant \, 2 \psi(y) - W^{\pa}_{k} , 
\end{align*}
where the constant $W^{\pa}_{k}$ is given by
\[
    W^{\pa}_{k}  \, = \,  \myfrac{1}{3} \Bigl( 
    \psi \bigl( y^{[2k]}1\overline{0}\ts \bigr) - \psi 
    \bigl( y^{[2k]}\overline{0}\ts \bigr) \Bigr) \, > \, 0 \ts .
\]
The first inequality now follows from the facts that the cylinder
$\langle y^{[2k-1]}0 \rangle$ is coarser than
$\langle y^{[2k-1]}11\rangle$ and that $\psi^{\ts\prime}$ is positive
and decreasing on $ \bigl[ 0, \frac{1}{2} \bigr]$.  The last step is
then a consequence of the concavity of $\psi$ together with the
observation that
$y = \frac{1}{3} y^{[2k]}\overline{0} + \frac{2}{3}
y^{[2k]}1\overline{0}$.
\end{proof}

\begin{coro}\label{coro:maxima-y}
  Let\/ $y$ be the alternating sequence from Eq.~\eqref{eq:def-y-hat-y}
  and let\/ $\varepsilon>0$ be the constant  from
  Lemma~\textnormal{\ref{lem:equiv-to-prev-est}}.  Then, for all\/
  $n\in\NN$ and\/ $z\in \XX$, one has
 \[
     \psi^{\pa}_{n} (y) \,\geqslant\, \psi^{\pa}_{n}
     \bigl(y^{[n]}\widehat{y}^{\pa}_{n+1} z\bigr)+\varepsilon.
 \]
\end{coro}

\begin{proof}
  The proof is similar to that of Corollary~\ref{Coro:alternating}.
  To prove our claim, we employ Lemma~\ref{lem:equiv-to-prev-est} for
  $n=1$.  Analogously as in \eqref{eq:psi2m}, we obtain
\begin{align*}
    \psi^{\pa}_{2m} \bigl(y\bigr) \, & = \,
    \left( \bigl( \psi + \psi^{(2)} \bigr) 
     + \bigl( \psi^{(3)} + \psi^{(4)} \bigr) + \ldots + \bigl( \psi^{(2m-1)} 
     + \psi^{(2m)} \bigr) \right) \bigl( y \bigr)
     \,= \, m \bigl( \psi + \psi^{(2)} \bigr) \bigl( y \bigr) \\[1mm] 
     & \geqslant \, \bigl( \psi + \psi^{(2)} \bigr) 
       \bigl( y^{[2m]}1z\ts \bigr) 
     + \bigl( \psi + \psi^{(2)} \bigr) \bigl( 
     y^{[2m-2]}1z \ts \bigr)
     + \ldots + \bigl( \psi + \psi^{(2)} \bigr)
     \bigl( y^{[2]}1z \ts \bigr)+\varepsilon \\[1mm]
     & = \, \psi^{\pa}_{2m} \bigl( y^{[2m]}1 z \ts \bigr)
        + \varepsilon \ts .
\end{align*}
From this result, we can also proceed to $n = 2m+1$ with $m \in \NN$,
\begin{align*}
    \psi^{\pa}_{2m+1} \bigl( y \bigr) \, & = \,
    \psi^{\pa}_{2m} \bigl( \widehat{y} \ts \bigr) + \psi \bigl( y \bigr)
    \,=\, \psi^{\pa}_{2m} \bigl( y \bigr) + \psi \bigl( y \bigr)
         \, \geqslant \, \psi^{\pa}_{2m} \bigl(
        y^{[2m]}1\widehat{z} \ts \bigr)+\varepsilon
       + \psi \bigl( y^{[2m+1]}0\ts z \ts \bigr)  \\[1mm] 
    & = \,\psi^{\pa}_{2m} \bigl( \widehat{y}^{\ts [2m]}0 z \ts \bigr)
       + \varepsilon + \psi \bigl( y^{[2m+1]}0\ts z \ts \bigr)
      \, = \, \psi^{\pa}_{2m+1} \bigl( y^{[2m+1]}0\ts z \ts \bigr)
      + \varepsilon \ts ,
 \end{align*}
which completes the argument.
\end{proof}

Next, we split the set $\Sigma^n$ into the sets $D_{\ell,n}$, where  
\[
     D_{\ell,n} \, := \, \Bigl\{ a^{r^{\pa}_1} I
      \bigl( a^{r^{\pa}_2} \bigr)
      \cdots I^{\ell-2} \bigl( a^{r^{\pa}_{\ell-1}}\bigr) 
      I^{\ell-1} \bigl(a^{r^{\pa}_{\ell}}\bigr) : a\in\{0, 1\}, 
      \, r^{\pa}_k \in\NN, \sum_{k=1}^{\ell} r^{\pa}_k = n \Bigr\}
\]
with $I$ defined as in Proposition~\ref{prop:pointwise-compare}.  Put
differently, we consider words of length $n$ that consist of $\ell$
blocks of consecutive $0$s and $1$s.  With the help of
Corollary~\ref{coro:maxima-y}, we can now prove the following result.

\begin{lemma}\label{lem:psi(Dnl)}
  There exist constants $K,\delta > 0$ such that
\[
     \max_{\omega\in D_{\ell,n}}\,
       \sup_{x\in\langle\omega\rangle}\psi^{\pa}_{n}(x) 
       \,\leqslant\, K + n \log(3/2)- \delta(n - \ell) 
\]
holds uniformly for all\/ $n,\ell\in\NN$ with $\ell\leqslant n$.
\end{lemma}

\begin{proof}
  Let $\omega\in D_{\ell,n}$, so
  $\card\{i\leqslant n :
  \omega^{\pa}_{i-1}=\omega^{\pa}_{i}\}=n-\ell$. We set
  $s(0):= s(0,\omega):= 0$, together with
  $s(k+1):= s(k+1,\omega):= \min\left\{i>s(k) :
    \omega^{\pa}_i=\omega^{\pa}_{i+1}\right\}$ for $k <n-\ell$ and
  $s(n-\ell+1):=n$. Then, we have
\[
      \sup_{x\in\langle\omega\rangle}\psi^{\pa}_{n}(x)
      \,\leqslant \sum_{k=1}^{n-\ell+1} \sup_{x\in\sigma^{s(k-1)}
      (\langle\omega\rangle)}\! \psi^{\pa}_{s(k)-s(k-1)}(x) \ts .
\]
By the definition of $s$ and the symmetry of $\psi$, we have  
\[
    \sup_{x\in\sigma^{s(k-1)} (\langle\omega\rangle)} \!
    \psi^{\pa}_{s(k)-s(k-1)}(x) \, \leqslant\, \sup_{x\in\XX}
    \psi^{\pa}_{s(k)-s(k-1)} \bigl(y^{[s(k)-s(k-1)]}
    \widehat{y}^{\pa}_{s(k)-s(k-1)+1} x \bigr) 
\]
for any $k\leqslant n-\ell$. Next, when $s(k)-s(k-1)=1$, we get
\[
       \sup_{x\in\sigma^{s(k-1)} (\langle\omega\rangle)} 
       \psi^{\pa}_{s(k)-s(k-1)}(x)
       \, \leqslant \, \sup_{x\in\XX} \psi (00x)
       \, = \, 0 \ts .
\]
     If $s(k)-s(k-1)>1$, Corollary~\ref{coro:maxima-y} implies
\[
    \sup_{x\in\sigma^{s(k-1)} (\langle\omega\rangle)} \! 
      \psi^{\pa}_{s(k)-s(k-1)}(x)
    \,  \leqslant \, \psi^{\pa}_{s(k)-s(k-1)}(y)-\varepsilon 
    \,  = \, (s(k)-s(k-1)) \log(3/2) - \varepsilon \ts .
\]
Finally, Corollary~\ref{coro:maxima} and Lemma~\ref{lem:const}, with
the constant $K>0$ defined there, give
\begin{align*} 
     \sup_{x\in\sigma^{s(n-\ell)} (\langle\omega\rangle)}
     \psi^{\pa}_{n-s(n-\ell)}(x)
     \, & = \, \sup_{x\in\XX}\psi^{\pa}_{n-s(n-\ell)}
     \bigl(y^{[n-s(n-\ell)]}x \bigr)
     \, = \, \sup_{x\in\XX} \psi^{\pa}_{n-s(n-\ell)} (x) \\[1mm]
     &\leqslant\,\sup_{x\in\XX} \psi^{\pa}_{n-s(n-\ell)}
        \bigl( \overline{01} \ts \bigr) + K
      \, = \, (n-s(n-\ell)) \log(3/2) + K \ts .
\end{align*}

Combining these five estimates and setting
$\delta = \min \left\{ \varepsilon, \log \bigl( \frac{3}{2}\bigr)
\right\}$ yields
\begin{align*}
     \sup_{x\in\langle\omega\rangle}\psi^{\pa}_{n}(x)
     \,&\leqslant\, \sum_{k=1}^{n-\ell} \bigl( (s(k)-s(k-1))
     \log(3/2)-\delta \bigr) 
      \; + (n-s(n-\ell) ) \log(3/2) +K\\[2mm]
    &=\, n \log(3/2)-(n-\ell)\delta+K \ts ,
\end{align*}
and thus the statement of the lemma.
\end{proof}

For the final step, we set $\gamma:= \log\bigl( \frac{3}{2}\bigr)$. Using
Lemma~\ref{lem:psi(Dnl)} together with
$\card \bigl(D_{\ell,n}\bigr) = \binom{n-1}{\ell-1}$ to justify the
first inequality gives
\begin{align*}
   \myfrac{1}{n} \log \sum_{\omega\in \Sigma^n}
   \, & \sup_{x\in\langle\omega\rangle}
     \exp(t\ts \psi^{\pa}_{n}(x))-t\gamma
   \, = \, \myfrac{1}{n} \log \sum_{\ell=1}^n
        \sum_{\omega\in D_{\ell,n}}
        \sup_{x\in\langle\omega\rangle}
        \exp(t\ts \psi^{\pa}_{n}(x))-t\gamma \\
  & \leqslant\, \myfrac{ tK}{n}+\myfrac{1}{n}\log \sum_{\ell=1}^n
        \binom{n-1}{\ell-1} \exp (-(n-\ell)\delta t ) \\[1mm]
  & =\, \myfrac{ tK}{n} + \myfrac{1}{n}\log (1+\exp
        (-\varepsilon t ))^{n-1}
  \,=\, \myfrac{ tK}{n}+\myfrac{n-1}{n}\log (1+\exp
         (-\delta t )) \ts .
\end{align*}
Letting first $n\to \infty$ and then $t\to \infty$ gives the desired
result.

Let us finally prove \eqref{eq:prop-2.5c}. To see that
$p^{\ts \prime}(0) \geqslant -\log(2)$, we note that, by
Theorem~\ref{thm:MAIN-MF-1}, the inequality
$p^{\ts \prime} (0) < -\log(2)$ would imply that $b(-\log(2))<1$,
which contradicts the fact that $\cB(-\log(2))$ as defined in
Eq.~\eqref{eq:Birk-spec-and-level-sets} has full Lebesgue measure.

Hence, it remains to show that
$p^{\ts \prime} (0) \leqslant -\log(2)$.  Recall the definition of
$\psi^{D(\ell)}$ from \eqref{eq:def-psi-m} and set
$\psi_{n}^{D(\ell)}(x):= \sum_{k=1}^n\psi^{D(\ell)} (2^{k-1}x )$,
where $W^{\pa}_{\nts\ell}$ denotes the bounded distortion constant for
$\psi^{D(\ell)}$.  For $\omega\in\Sigma^{n}$, we define
$\overline{\psi^{\pa}_{n,\omega}}:= \sup_{x\in\langle \omega
  \rangle}\psi^{\pa}_{n}(x)$ and
$\underline{\psi^{\pa}_{n,\omega}}:= \inf_{x\in\langle \omega
  \rangle}\psi^{\pa}_{n}(x)$, and analogously
$\overline{\psi^{D(\ell)}_{n,\omega}}$ and
$\underline{\psi^{D(\ell)}_{n,\omega}}$.
Eq.~\eqref{eq:PressureByInf-1} and the properties of the H\"older mean
then imply that, for all $n,\ell \in \NN$,
\begin{align*}
    p^{\ts \prime} (0+) \, & \leqslant \, 
     \inf_{t\geqslant 0} \frac{  {n^{-1}} \log \sum_{\omega\in \Sigma^{n}}
     \exp \bigl( t \overline{\psi^{\pa}_{n,\omega}}\ts \bigr)
            -\log(2) }{t} \\[1mm] 
   &  = \, \inf_{t\geqslant 0} \myfrac{1}{n}\log 
   \Bigl( 2^{-n}\sum_{\omega\in\Sigma^{n}} 
   \exp \bigl(\ts\overline{\psi^{\pa}_{n,\omega}}\ts\bigr)^{t}  \Bigr)^{1/t}
   \, = \, \myfrac{1}{n}\log \biggl( \, \prod_{\omega\in\Sigma^{n}}
     \exp \bigl(\ts\overline{\psi^{\pa}_{n,\omega}} \ts\bigr)
     \! \biggr)^{2^{-n}} \\[1mm] 
         & \, = \, \myfrac{1}{n\ts 2^{n}}
           \sum_{\omega\in\Sigma^{n}} \overline{\psi^{\pa}_{n,\omega}}
           \, \leqslant \, \myfrac{1}{n\ts 2^{n}} \sum_{\omega\in\Sigma^{n}}
           \overline{\psi^{D(\ell)}_{n,\omega}} \, \leqslant \,
           \myfrac{1}{n\ts 2^{n}} \sum_{\omega\in\Sigma^{n}}
           \Bigl(\ts \underline{\psi^{D(\ell)}_{n,\omega}} \,
            + W^{\pa}_{\! \ell} \! \nts \Bigr)\\[1mm]
         & \, \leqslant \, \myfrac{1}{n}\int
           \psi_{n}^{D(\ell)} \dd\lambda \,
           + \frac{W^{\pa}_{\!\ell}}{n} \, =
           \int \psi^{D(\ell)} \dd\lambda \,
           +\frac{W^{\pa}_{\!\ell}}{n} \ts .
\end{align*}
Now, letting first $n$ tend to infinity and then $\ell$ gives the
desired inequality.
\end{proof}

\begin{remark}
  The graph of the pressure function in
  Figure~\ref{fig:The-graph-of-pressureFunction-1} was generated using
  approximants of the form
\[
    p^{[n]}(t) \, = \, \myfrac{1}{n-2} \, \log 
    \sum_{j = 1}^{2^{n-2}} \Bigl( \myfrac{1}{2} \exp{
     \psi^{\pa}_{n} \bigl((2j -1) 2^{-n} \bigr) } \Bigr)^t ,
\]
which can be shown to converge to $p(t)$ by an explicit calculation.
Thus, we have reduced the number of evaluations on cylinders by a
factor of $4$ using symmetry, and we have avoided singular points. The
modification from $n$ to $n-2$ in the denominator and the additional
factor $\frac{1}{2}$ in front of the exponential ensure the
normalisations $p^{[n]}(0) = \log(2)$ and $p^{[n]}(1) = \log(2)$ for
all $n$, respectively.  It turns out that these approximants exhibit a
considerably faster convergence than approximants without the above
normalisations. Since the computational cost increases exponentially
in $n$, this is practically relevant. Note that this approach is
equivalent to the numerics presented in \cite{GL90}.  \exend
\end{remark}

\section{The remaining parts of the proof of 
   Theorem~\ref{thm:MAIN-MF-1}}\label{Sec:Proof-of-Thm-II}
 
 Proposition~\ref{prop:poinwise-conv} enables us to prove the
 remaining claims of Theorem~\ref{thm:MAIN-MF-1}.  We are left to show
 that the Birkhoff spectrum $b$ satisfies the following properties:
 \begin{itemize}\itemsep=2pt
 \item[(B1)] The function $\birk$ is concave on\/
   $(-\infty,\log(3/2)]$ and vanishes in the right boundary
   point of this interval.
 \item[(B2)] The level sets\/ $\cB(\alpha)$ are empty for
    $\alpha > \log(3/2)$.
 \item[(B3)] We have $b(\alpha)=1$ for $\alpha \leqslant - \log(2)$  and 
     $b(\alpha)<1$ for $\alpha > - \log(2)$.
 \end{itemize}
 \smallskip
 
 \noindent(B1): By Proposition~\ref{prop:poinwise-conv} the pressure
 function $p$ is convex and thus its Legendre transform $p^{*}$ is
 convex on its domain of definition given by
 $\left\{ x\in \RR\colon \sup_{t\in \RR} (xt-p(t)) <\infty \right\}$
 which is equal to $(-\infty,\log(3/2)]$. Further, as a
 consequence of \eqref{eq:prop-2.5a}, we have
\[
     -p^{*} \bigl(\log(3/2)\bigr) \,=\, \inf_{t\in\RR} 
     \bigl(p(t) - \log(3/2)\, t \bigr) \,=\, 0 \ts ,
\]
and it follows that  $b\bigl(\log \bigl( \frac{3}{2} \bigr) \bigr)=0$. 
\smallskip
  
\noindent(B2): The fact that the level sets are empty for
$\alpha>\log\bigl( \frac{3}{2}\bigr)$ follows from
Lemma~\ref{maxBeta}.  \smallskip
 
\noindent(B3): Clearly, $\birk(\alpha) \leqslant 1$ for all
$\alpha \in \RR$ by definition.  On the other hand, for
$\alpha\leqslant -\log(2)$, we have
\[
     \birk(\alpha) \, \geqslant \,
     \frac{-p^{*}(\alpha)}{\log(2)}
     \,=\,\frac{\inf_{t\geqslant 0} (p(t)-\alpha t)}{\log(2)}
     \,\geqslant\, \frac{\inf_{t\geqslant 0} 
     (p(0)-t\log(2)-\alpha t)}{\log(2)} \,=\,1 \ts ,
\]
where we have used the fact that $p$ is convex on $[0,\infty)$
together with Eqs.~\eqref{eq:prop-2.5aa} and \eqref{eq:prop-2.5c}.
This proves the first claim. The second is an immediate consequence of
the definition fo the Legendre transform in conjunction with
Eqs.~\eqref{eq:prop-2.5aa} and \eqref{eq:prop-2.5c}.

\section{Variational principle and equilibrium
    measure}\label{Sec:Eq-measure}

  So far, we have employed various concepts from the thermodynamic
  formalism and shown explicitly some results that, in general, were
  known to hold only for a more restrictive class of potentials. In
  particular, the relationship between the Birkhoff spectrum and the
  scaling or
  $f (\alpha)$-spectrum is of the same form as the one holding for
  equilibrium measures of H{\"o}lder continuous potentials
  \cite[Prop.~1]{PesinWeiss}.  It is the aim of this section to show
  that our results embed nicely into the known formalism in the sense
  that $\nu$ is indeed an {equilibrium measure} for the potential
  $\psi$. Let us expand a bit on this concept. To an upper
  semi-continuous function $\phi$ on the compact dynamical system
  $(X,T)$, we assign its \emph{variational pressure}
\[
    \cP^{\pa}_T(\phi) \, = \sup_{\mu \in \cM_T} 
    \bigl( h(\mu) + \mu(\phi) \bigr),
\]
where $\cM^{\pa}_T$ denotes the set of $T$-invariant probability
measures on $X$ and $h(\mu)$ is the metric entropy of $\mu$; compare
\cite{Keller}. A measure $\mu \in \cM^{\pa}_T$ that maximises the
above expression is called an \emph{equilibrium measure}.

In fact, our measure $\nu$ satisfies the even stronger property of
being a $g$-measure in the sense of Keane \cite{keane71}. We call a
non-negative function on $X$ a $g$-function if
\[
     \sum_{x \in T^{-1}(y)} \! g(x) \, = \, 1 
\]
holds for all $y \in \XX$. The associated operator $\varphi_g$ on the
space of bounded measurable functions $B(X)$ is defined via
\[
  \bigl( \varphi_g f \bigr)(y) \, = \!
  \sum_{x \in T^{-1} (y)} \! g(x) f(x)
\]
and, for any $\mu \in \cM^{\pa}_T$, its dual operation is given by
$\bigl( \varphi_g^* \ts \mu \bigr)(f) = \mu(\varphi^{\pa}_g f)$.  Any
Borel probability measure $\mu$ with $\varphi_g^* \ts \mu = \mu$
necessarily lies in $\cM^{\pa}_T$ and is called a $g$-measure.

It is a matter of direct calculation to verify that the function
$g(x) = \frac{1}{2} \bigl(1 - \cos(2 \pi x)\bigr)$ is indeed a
$g$-function for the dynamical system $(\XX,\sigma)$, as well as for
$(\TT ,T)$.  Similarly, we find that $\nu$ is a $g$-measure with this
choice of $g$ by a straight-forward computation.

The following characterisation of $g$-measures is classic.

\begin{fact}[{\cite[Thm.~1]{Ledrappier}}]\label{fact:ledrappier}
  Let\/ $g$ be a\/ $g$-function and\/ $\mu$ a probability measure on\/
  $(\XX,\sigma)$.  Then, the following characterisations are
  equivalent.
\begin{enumerate}\itemsep=2pt
  \item One has\/ $\varphi_g^* \ts \mu = \mu$.
  \item The measure\/ $\mu$ satisfies\/ $\mu \in \cM_{\sigma}$ and
  is an equilibrium measure for the potential\/ $\log(g)$, with\/
   $\, \cP_{\sigma}(\log(g))  =  h(\mu) + \mu(\log (g))  = 0 $.  \qed
\end{enumerate}
\end{fact}

Note that $ \log (g) = \psi - \log(2)$. We therefore find
\[
  \cP_{\sigma}(\psi) \, = \sup_{\mu \in \cM^{\pa}_{\sigma}} 
  \bigl(h(\mu) + \mu( \log (g)) + \log(2)\bigr) \, = \,
   \cP_{\sigma}( \log (g)) + \log(2) \, = \, \log(2) \ts ,
\]
and a measure is an equilibrium measure for $\psi$ if and only if it
is an equilibrium measure for $\log (g)$ which, in turn, is true
 if and only if it is a
$g$-measure. In \cite[Section~2]{keane}, it was shown that the
$g$-measure on the dynamical system $(\TT,T)$ with
$T \colon x \mapsto 2x$ (mod $1$), is unique and strongly mixing if
$g=0$ at a single position in $[0,1)$.  For a detailed discussion on
the existence of $g$-measures for Riesz products, we refer to
\cite{fan-g-measures}.

\begin{coro}\label{coro:TM-pressure}
  The Thue{\ts}--Morse measure\/ $\nu$ is a strongly mixing
  $g$-measure. In particular, it is the unique equilibrium measure for
  the potential $\psi$ and
\[
    \cP^{\pa}_{\sigma}(\psi) \, = \, h(\nu) + \nu(\psi)
    \, = \, \log(2) 
\]
  is the corresponding variational pressure.   \qed
\end{coro}

In fact, by using Eq.~\eqref{eq:prop-2.5aaa}, it turns out that the
notions of variational pressure and topological pressure of $\psi$
coincide. This is well known for H\"{o}lder continuous potentials
\cite[Thm.~1.22]{Bow}.  The relation $h(\nu) + \nu(\psi) = \log (2)$
was also established in \cite{Q1}, using slightly different
techniques.

\begin{remark}
  The uniqueness of $\nu$ as an equilibrium measure can be obtained by
  generalising the Ruelle--Perron--Frobenius theorem \cite{ruelle} to
  our potential $\psi$. This provides additional information about the
  operator $\varphi^{\pa}_g$, such as the fact that $\mathbf{1}$ is
  the (up to normalisation) unique eigenfunction to the eigenvalue $1$
  and the fact that the remainder of the spectrum is contained in a
  disk around zero of radius strictly smaller than $1$.  \exend
\end{remark}

\begin{remark}
  It is worth noticing that Corollary~\ref{coro:TM-pressure} permits
  us to give a closed form for the metric entropy of $\nu$ in terms of
  the autocorrelation coefficients; compare \cite[Eqs.~(19) and
  (20)]{Zaks}. There, the authors established the relation
\begin{equation}\label{Eq:expression-for-energy}
    \lim_{n \rightarrow \infty} \nu_n(\psi) \, = \, - \log(2) 
       - 2 \sum_{j =1}^{\infty} \frac{\eta(j)}{j} \ts ,
\end{equation}
where $\eta(j) =  \lim_{n \rightarrow \infty} \frac{1}{n}
\sum_{k = 1}^{n} v^{\pa}_k v^{\pa}_{k+j}$ are the autocorrelation
coefficients, with $(v^{\pa}_{k})^{\pa}_{k\in \NN}$ denoting the one-sided
Thue{\ts}--Morse sequence in $\{ \pm 1\}$ that starts with $1$; see
\cite{TAO} for background. Note that the relation
\[
    \nu(\psi)  \, =  \lim_{n \rightarrow \infty} \nu^{\pa}_{n} (\psi)
\]
can be checked easily. This is due to the fact that the singularities
in $\psi$ are relatively `weak' and the measures $\nu^{\pa}_n$ assign
rapidly decreasing values to boundary intervals of $[0,1]$.

Due to the well-known renormalisation equations for $\eta$, 
compare \cite[Sec.~10.1]{TAO}, the series
in Eq.~\eqref{Eq:expression-for-energy} can be numerically evaluated
with high precision. Following the procedure presented in \cite{Zaks},
we find via Corollary~\ref{coro:TM-pressure} that
\[
    h(\nu) \, = \, 2\log(2) + 2 \sum_{j =1}^{\infty} \frac{\eta(j)}{j} 
    \, \approx \, 0.506{\,}383{\,}995{\,}447{\,}319{\,}674{\,}30 \ts ,
\]
with a precision of $20$ correct digits. Note that this value is
related to the information dimension $D_1$, as calculated in
\cite{Zaks} and \cite{GL90}, via $h(\nu) = \log(2) D_1$, so
$D_1 \approx 0.730$. Our numerical value is a significant improvement
over the lower bound for the entropy derived in \cite{PK}.  It is
precise enough to affirmatively answer the question from
\cite[Sec.~4.4.1]{Q1} whether the (information) dimension of $\nu$ is
strictly larger than its \emph{energy exponent}
$e(\nu) := 1- \log_2(\kappa)$, with
$\kappa = \bigl(1 + \sqrt{17} \,\bigr)/4$, which gives
$e(\nu) \approx 0.643$.  \exend
\end{remark}

\section*{Acknowledgements} 

MK acknowledges support from the German Research Foundation (DFG),
through grant KE 1440/3-1, and would like to thank the Mittag--Leffler
institute for its kind hospitality during the research program
\textsl{Fractal Geometry and Dynamics}, where valuable discussions
with F.~Ekstr\"om and J.~Schmeling took place.  This work was also
supported by the Research Centre of Mathematical Modelling (RCM$^2$)
of Bielefeld University (for TS), and by the DFG Collaborative
Research Centre 1283 at Bielefeld (for MB and PG).

\end{document}